\documentclass[12pt]{amsart}
\usepackage{amssymb}
\usepackage[all]{xy}

\input xy
\xyoption{all}

\newtheorem{lemma}{Lemma}[section]
\newtheorem{corollary}[lemma]{Corollary}
\newtheorem{theorem}[lemma]{Theorem}
\newtheorem{proposition}[lemma]{Proposition}

\theoremstyle{definition}
\newtheorem{remark}[lemma]{Remark}
\newtheorem{definition}[lemma]{Definition}

\newtheorem{example}[lemma]{Example}

\newcommand\Cset{\mathbb {C}}
\newcommand\Zset{\mathbb {Z}}

\newcommand\Qset{\mathbb {Q}}

\newcommand\ve{{\mathcal V}}
\newcommand\A{{\mathcal A}}
\newcommand\U{{\mathcal U(A)}}
\newcommand\ann{\mathrm{ann}}
\newcommand\Qrtot{Q^r_{\mathrm{tot}}}
\newcommand\Qltot{Q^l_{\mathrm{tot}}}
\newcommand\Qtot{Q^{\sigma}_{\mathrm{tot}}}

\newcommand\Qrmax{Q^r_{\mathrm{max}}}
\newcommand\Qlmax{Q^l_{\mathrm{max}}}
\newcommand\Qmax{Q^{\sigma}_{\mathrm{max}}}
\newcommand\Qrcl{Q^r_{\mathrm{cl}}}
\newcommand\Qlcl{Q^l_{\mathrm{cl}}}

\newcommand\sima{\sim^{\hskip-.22cm a}}
\newcommand\sims{\sim^{\hskip-.22cm *}}
\newcommand\manjea{\preceq^{\hskip-.25cm a}}
\newcommand\manjes{\preceq^{\hskip-.25cm *}}

\newcommand\ef{{\mathfrak F}}
\newcommand\dirlim{\mathop{\varinjlim}\limits}
\newcommand\f{{\mathcal F}}
\newcommand\te{{\mathcal T}}

\newcommand\cl{\mathrm{cl}}

\newcommand\homo{\mathrm{Hom}}

\newcommand\unb{\mathrm{{\bf u}}}
\newcommand\bnd{\mathrm{{\bf b}}}
\newcommand\smallt{\mathrm{{\bf t}}}
\newcommand\p{\mathrm{{\bf p}}}
\newcommand\T{\mathrm{{\bf T}}}
\newcommand\bigP{\mathrm{{\bf P}}}

\begin{document}
\title{Strongly semihereditary rings and rings with dimension}

\author{Lia Va\v s}
\address{Department of Mathematics, Physics and Statistics\\
University of the Sciences in Philadelphia\\
Philadelphia, PA 19104, USA}
\email{l.vas@usp.edu}

\thanks{This paper was completed mostly during the author's visit to the University of M\'alaga, partially funded by a ``Grant for foreign visiting professors'' within the III Research Framework Program of the University of M\'alaga. The author thanks the host institution for the hospitality and support.}




\begin{abstract} The existence of a well-behaved dimension of a finite von Neumann algebra (see \cite{Lueck_dimension}) has lead to the study of such a dimension of finite Baer $*$-rings (see \cite{Lia_Baer}) that satisfy certain $*$-ring axioms (used in \cite{Berberian}). This dimension is closely related to the equivalence relation $\sims$ on projections defined by $p\sims q$ iff $p=xx^*$ and $q=x^*x$ for some $x.$ However, the equivalence $\sima$ on projections (or, in general, idempotents) defined
by $p\sima q$ iff $p=xy$ and $q=yx$ for some $x$ and $y,$ can also be relevant. There were attempts to unify the two approaches (see \cite{Berberian_web}).

In this work, our agenda is three-fold: (1) We study assumptions on a ring with involution that guarantee the existence of a well-behaved dimension defined for any general equivalence relation on projections $\sim.$ (2) By interpreting $\sim$ as $\sima,$ we prove the existence of a well-behaved dimension of strongly semihereditary rings with a positive definite involution. This class is wider than the class of finite Baer $*$-rings with dimension considered in the past: it includes some non Rickart $*$-rings. Moreover, none of the $*$-ring axioms from \cite{Berberian} and \cite{Lia_Baer} are assumed. (3) As the first corollary of (2), we obtain dimension of noetherian Leavitt path algebras over positive definite fields. Secondly, we obtain dimension of a Baer $*$-ring $R$ satisfying the first seven axioms from \cite{Lia_Baer} (in particular, dimension of finite $AW^*$-algebras). Assuming the eight axiom as well, $R$ has dimension for $\sims$ also and the two dimensions coincide.

While establishing (2), we obtain some additional results for a right strongly semihereditary ring $R$: we prove that every finitely generated $R$-module $M$ splits as a direct sum of a finitely generated projective module and a singular module; we describe right strongly semihereditary rings in terms of relations between their maximal and total rings of quotients; and we characterize extending Leavitt path algebras over finite graphs.
\end{abstract}

\maketitle

\section{Introduction}
\label{section_introduction}

The book \cite{Berberian} summarizes the trend to algebraize the operator theory concepts (e.g. von Neumann algebras) initiated by J. von Neumann and I. Kaplansky. In the introduction, S. K. Berberian explained the motivating idea by saying that ``Von Neumann algebras are blessed with an excess of structure -- algebraic, geometric, topological -- so much, that one can easily obscure, through proof by overkill, what makes a particular theorem work'' and that ``if all the functional analysis is
stripped away ... what remains should (be) completely accessible through algebraic avenues.''

The paper \cite{Lia_Baer} was motivated exactly by these ideas. It is not an overstatement to say that the current paper follows Berberian's ideas to an even larger extent. Namely, our main goal is to generalize some properties of finite von Neumann algebras to a class of rings even wider than that considered in \cite{Lia_Baer}. The rings in \cite{Lia_Baer} are Baer $*$-rings that satisfy nine axioms (later reduced to eight, see \cite{Lia_nine}), some of which are particularly restrictive. Here, we prove that results shown in \cite{Lia_Baer} hold for certain rings that are not necessary even Rickart $*$-rings. Also, in most cases, we do not assume any of the axioms from \cite{Lia_Baer}.

Besides considering more general classes of rings, we also consider another level of generality. Namely, the rings considered in \cite{Berberian} and \cite{Lia_Baer} are finite in the sense that $p\sims 1$ implies $p=1$ for all projections $p.$ The equivalence relation $\sims$ used here is defined by $p\sims q$ iff $p=xx^*$ and $q=x^*x$ for some $x.$  However, there is another equivalence relation, sometimes referred to as algebraic equivalence, defined for idempotents by $p\sima q$ iff $p=xy$ and $q=yx$ for some $x$ and $y.$ The concept of ``finiteness'' in this ``non-$*$'' setting means that  $p\sima 1$ implies $p=1$ for all idempotents $p$ and the ring is said to be directly finite (or Dedekind finite) in this case. A good portion of \cite{Berberian_web} is an attempt to unify these two approaches by studying a general relation $\sim.$ Following this idea, we formulate a set of general assumptions on a ring and a general relation $\sim$ that yield a dimension function on projections and  dimension of all $R$-modules (Theorem \ref{general_dim_on_R}). The dimension of modules is well-behaved (in the sense of Theorem \ref{LueckDimension}). In particular, it is additive.

As opposed to concepts like the rank of right noetherian or semiprime Goldie rings, this dimension is defined for rings that are not necessarily noetherian or Goldie. The Goldie reduced rank (defined as the uniform dimension of the quotient of a module and the Goldie closure of its trivial submodule) is defined for any ring so it also does not have the setback of the ranks we just mentioned. However, the Goldie reduced rank fails to distinguish different uniform modules or different modules decomposable into a sum of infinitely many direct summands. Although not completely general, the dimension we consider does not have these lapses.

The dimension of a ring $R$ as in Theorem \ref{general_dim_on_R} is defined via dimension of a regular overring $Q$ of $R$ such that the direct summands of $Q^n$ are in one-to-one correspondence with the direct summands of $R^n.$
The inspiration to consider a regular overring came from the fact that both in the case of finite von Neumann algebras and in the case of row-finite Leavitt path algebras, certain regular overring with some favorable properties is defined. In both cases, these regular overrings turn out to be the maximal right rings of quotients.
\begin{enumerate}
\item In case of a finite von Neumann algebra $\A$, the algebra of affiliated operators $\U$ is a regular overring of $\A$ (see \cite{Lueck} for details). In \cite[Chapter 6]{Berberian}, Berberian generalized this construction to a construction of a regular ring $Q$ of a Baer $*$-ring $R$ satisfying certain $*$-ring axioms, denoted by (A1)--(A7) in \cite{Lia_Baer}. By \cite[Proposition 3]{Lia_Baer}, the regular ring $Q$ is the maximal right ring of quotients of $R.$

\item  In case of a row-finite Leavitt path algebra $L(E)$, the regular ring $Q(E)$ is defined in \cite{Ara_Brustenga}. By \cite[Proposition 5.1]{Gonzalo_Lia_noetherian}, the regular ring $Q(E)$ of a noetherian Leavitt path algebra $L_K(E)$ over a positive definite field $K$ is the maximal right ring of quotients of $L_K(E).$
\end{enumerate}

The common properties of both of these two classes of rings are encompassed by considering strongly semihereditary rings.
A right nonsingular ring is right strongly semihereditary if every finitely generated nonsingular right module is projective (equivalently, extending). We show that a well-behaved dimension exists for a strongly semihereditary ring with a positive definite involution (Theorem \ref{strongly_semihereditary_has_dim}) for $\sima.$

The class of right strongly semihereditary rings is quite versatile: there are finite $AW^*$-algebras that are neither noetherian nor hereditary and there are noetherian Leavitt path algebras that are not Rickart $*$-rings. This motivates the study of right strongly semihereditary rings in their own right.

The paper is organized as follows. In Section \ref{section_torsion_theories}, we recall some known facts and definitions. In Section \ref{section_strongly_semihereditary}, we focus on the class of right strongly semihereditary rings. In particular, we relate the conditions on extendability of certain classes of modules with the conditions stating that some rings of quotients coincide (Proposition \ref{right_strongly_semihereditary}).

For a right strongly semihereditary ring $R,$ the direct summands of $R^n$ and $\Qrmax(R)^n$ are in a bijective correspondence. This bijection induces the isomorphism of the monoids of equivalence classes of finitely generated projective modules of $R$ and $\Qrmax(R)$ (Proposition \ref{K0_theorem}). We also prove that that every finitely generated $R$-module $M$ splits as a direct sum of a finitely generated projective module and a singular module (Theorem \ref{SplittingOfBnd}).

In Section \ref{section_involutive_strongly_semihereditary}, we turn to strongly semihereditary rings with involution. We characterize them in terms of the relations between their maximal and total rings of quotients (Proposition \ref{strongly_semihereditary_involutive_rings}) and describe some additional favorable properties of this class of rings (Proposition \ref{properties}). As a corollary of this characterization we prove that a Leavitt path algebra over a finite graph and a positive definite field is extending if and only if the graph is no-exit (Corollary \ref{extending_LPAs}).

In Section \ref{section_dimension_general} we first demonstrate that certain regular $*$-rings have dimension defined for a general equivalence relation $\sim$ on projections of the matrix rings (Proposition \ref{general_dim_on_Q}). The existence of dimension of a regular overring of certain semihereditary rings enables us to define dimension of those semihereditary rings as well (Theorem \ref{general_dim_on_R}).

In Section \ref{section_dimension_strongly_semihereditary}, we use Theorem \ref{general_dim_on_R} to prove that dimension exists for positive definite strongly semihereditary $*$-rings 
when $\sim$ is interpreted as $\sima$ (Theorem \ref{strongly_semihereditary_has_dim}). As corollaries of this fact, we obtain two results. First, in Corollary \ref{LPAs}, we obtain the dimension of noetherian Leavitt path algebras over positive definite fields. Second, in Corollary \ref{VNA-like}, we obtain dimension of a Baer $*$-ring $R$ satisfying axioms (A1)--(A7) (as defined in \cite{Lia_Baer}).

If a Baer $*$-ring $R$ satisfies axiom (A8) in addition to (A1)--(A7), then $M_n(R)$ is a finite Baer $*$-ring that satisfies generalized comparability axiom (by \cite[Chapter 9]{Berberian} and \cite{Lia_nine}) and has dimension for $\sims$ (by \cite{Lia_Baer}). In this case,
Corollary \ref{VNA-like} also shows that $\sima$ and $\sims$ are the same relation on all matrix algebras $M_n(R)$ and that, as a consequence, the dimension obtained for $\sima$ is the same as dimension obtained in \cite{Lia_Baer} for $\sims.$

In \cite[Chapter 9]{Berberian}, Berberian refers to axioms (A8) and (A9) as ``unwelcome guests''. By \cite{Lia_nine}, (A9) follows from (A1)--(A7). We note that (A8) is not needed in order to obtain results from \cite[Section 58]{Berberian} -- the finiteness of the matrix ring $M_n(R)$ follows from (A1)--(A7) alone. It remains open whether (A8) also follows from (A1)--(A7).

\section{Torsion theories, closures, rings of quotients}
\label{section_torsion_theories}

In this paper, a ring is an associative ring with identity.
Throughout this paper, we use the definition of torsion theory, hereditary and faithful torsion theory, and Gabriel filter  as given in \cite{Stenstrom}. If $\tau = (\te, \f)$ is a torsion theory with torsion class $\te$ and torsion-free class $\f$ and if $M$ is a right $R$-module, the module $\te M$ denotes the torsion submodule of $M$ and $\f M$ is the torsion-free quotient $M/\te M.$ If, in addition, $K$ is a submodule of $M,$ the {\em closure} $\cl_{\te}^M(K)$ of $K$ in $M$ with respect to $\tau$ is $\cl_{\te}^M(K) = \pi^{-1}(\te(M/K))$  where $\pi$ is
the natural projection $M\twoheadrightarrow M/K.$ We suppress the superscript $M$ from $\cl_{\te}^M(K)$ if it is clear from the context in which module the closure has been taken. If $K$ is equal to its closure in $M,$ $K$ is said to be {\em closed} submodule of $M$. Note that $\te(M/K) = \cl_{\te}(K)/K$ (thus $\te M = \cl_{\te}(0)$ and $\f(M/K)\cong M/\cl_{\te}(K)$), and that $\cl_{\te}(K)$ is the smallest closed submodule of $M$ containing $K.$

If $\tau_1 = (\te_1, \f_1)$ and $\tau_2 = (\te_2, \f_2)$ are two torsion theories, $\tau_1$ is {\em smaller} than
$\tau_2$ ($\tau_1\leq\tau_2$) iff $\te_1\subseteq\te_2$ (equivalently $\f_1\supseteq\f_2$).

If $\tau$ is a hereditary torsion theory with Gabriel filter $\ef
= \ef_{\tau}$ and $M$ is a right $R$-module, the {\em module
of quotients} of $M$ is defined by:
\[M_{\ef} = \dirlim_{I\in\ef}\homo_R(I, M/\te M).\] The module $R_{\ef}$ has a
ring structure and is called the {\em
right ring of quotients}. The module $M_{\ef}$ is a right
$R_{\ef}$-module. Any additional necessary background can be found in \cite{Stenstrom}.

The {\em localization map} $q_M:M\rightarrow M_{\ef}$ (the composition of
$M\twoheadrightarrow M/\te M$ with $M/\te M\hookrightarrow M_{\ef}$) defines a left exact functor $q$ from the category of right $R$-modules to the category of right $R_{\ef}$-modules (see \cite[p. 197--199]{Stenstrom}). The kernel and cokernel of $q_M$ are torsion modules for every $M.$

We recall some examples of torsion theories and rings of quotients that we use in this paper.

\subsection{Examples.}
\label{Examples}

\subsubsection{ }\label{Lambek_example} All dense right ideals constitute a Gabriel filter. The corresponding torsion theory is called the {\em Lambek torsion theory} $\tau_L.$ It is the largest hereditary faithful torsion theory. The maximal right ring of quotients $\Qrmax(R)$ is the right ring of quotients with respect to $\tau_L.$

\subsubsection{ }\label{Goldie_example} The class of nonsingular modules over a ring $R$ is closed
under submodules, extensions, products and injective envelopes.
Thus, it is a torsion-free class of a hereditary torsion theory.
This theory is called the {\em Goldie torsion theory}.  For description of the closure operator in this torsion theory see \cite[Definition 7.31]{Lam}).

The Lambek theory is smaller than the Goldie theory. If $R$ is right nonsingular, then the Lambek and Goldie theories coincide. Following the notation from \cite{Lia_Baer}, the Goldie (= Lambek) torsion theory is denoted by $(\T, \bigP)$ in this case.

By \cite[Corollary 7.44$'$]{Lam}, there is a one-to-one correspondence between Goldie closed submodules of a nonsingular module $M$ of any ring $R$ and the Goldie closed submodules of the injective envelope $E(M)$ (for the definition of closed submodules, see \cite[Definition 7.31]{Lam}). By the remark after the proof of \cite[Corollary 7.44$'$]{Lam}, the closed submodules of $E(M)$ are precisely the direct summands of $E(M).$ Moreover, if $N$ is a closed submodule of $M$, then its closure in $E(M)$ is a copy of the injective envelope $E(N).$ By \cite[Proposition 7.44]{Lam}, a submodule $N$ of $M$ is closed if and only if $N$ is a complement in $M$. This gives us a one-to-one correspondence
 \[\{\mbox{complements in }M\} \longleftrightarrow\{\mbox{direct summands of }E(M)\}\] given by
$N \mapsto$ the closure of $N$ in $E(M)$ that is equal to a copy of $E(N).$ The inverse map is given by $L\mapsto L\cap M$ (see the proof of \cite[Corollary 7.44$'$]{Lam}). In particular, for a right nonsingular ring $R$ this gives us a one-to-one correspondence between Goldie (or Lambek) closed submodules of $R^n$ and direct summands of $\Qrmax(R)^n.$

\subsubsection{ }\label{tensoring_example} Let $R$ be a subring of a ring $S$ and let $S$ be flat as a left $R$-module. The collection of all $R$-modules $M$ such that $M\otimes_R S = 0$ defines a torsion class of a hereditary and faithful torsion theory denoted by $\tau_S.$ This theory is contained in the Lambek torsion theory.

\subsubsection{ }\label{perfect_example} A ring of right quotients $S$ of $R$ is said to be perfect if $S\otimes_R S\cong S$ and $S$ is left $R$-flat (see \cite[Ch. 11]{Stenstrom} for more details). A hereditary torsion theory $\tau$ with Gabriel filter $\ef$ is
called {\em perfect} if the right ring of quotients $R_{\ef}$ is perfect and $\ef=\{I| q_I(I)R_{\ef}=R_{\ef}\}$ where $q$ is the localization map. Also, if a ring of quotients $S$ is perfect, then $S$ is the ring of quotients of $R$ with respect to $\tau_S.$

Every ring has a maximal perfect right ring of quotients, unique up to isomorphism (\cite[Theorem XI 4.1]{Stenstrom}). It is
called {\em total right ring of quotients} and is denoted by $\Qrtot(R).$ By \cite[Theorem 12]{Lia_quotients}, if $R$ is right semihereditary, then $\tau_{\Qrmax}$ is a perfect torsion theory with the ring of quotients $\Qrtot.$ Following the notation in \cite{Lia_Baer}, we denote the torsion theory $\tau_{\Qrmax}$ also by $(\smallt, \p).$

\subsubsection{ }\label{bounded_example}
If $R$ is any ring, let $(\bnd, \unb)$ be the torsion theory in which a module $M$ is in $\bnd$ iff $\homo_R(M,R)=0$. We call a module in $\bnd$ a {\em bounded module} and a module in $\unb$ an {\em unbounded module}. This theory is the largest torsion theory in which $R$ is
torsion-free (thus it is larger than Lambek torsion theory). It is not necessarily hereditary.

For a right semihereditary ring we have
\[\mbox{trivial }\leq (\smallt, \p)\leq (\T, \bigP) \leq (\bnd, \unb)\leq \mbox{ improper}.\]
\cite[p. 24]{Lia_Baer} contains references for examples showing that all the inequalities can be strict.

\subsection{Symmetric rings of quotients}
If $\ef_l$ and $\ef_r$ are left and right Gabriel filters, the symmetric filter $_l\ef_r$ induced by $\ef_l$ and $\ef_r$ is defined to be the set of (two-sided) ideals of $R$ containing ideals of the form $IR+RJ$, where $I\in\ef_l$ and $J\in \ef_r$ (equivalently, the set of right ideals of $R\otimes_{\Zset}R^{op}$ containing ideals of the form $J\otimes R^{op}+R\otimes I$). The torsion theories corresponding to $\ef_l,$  $\ef_r,$ and $_l\ef_r$ are denoted by $\tau_l,$ $\tau_r,$ and $_l\tau_r$ respectively.  The symmetric ring of quotients $_{\ef_l}R_{\ef_r}$ (or $_lR_r$ for short) with respect to $_l\ef_r$ is defined by \[_{\ef_l}R_{\ef_r}=\dirlim_{K\in _l\ef_r}\;  \homo(K, \frac{R}{_l\te_r R})\] where the homomorphisms in the formula are $R$-bimodule homomorphisms (see \cite{Ortega_paper}).

If $\ef_l$ is the filter of dense left and $\ef_r$ the filter of dense right ideals, the symmetric ring of quotients with respect to the filter induced by $\ef_l$ and $\ef_r$ is the maximal symmetric ring of quotients  $\Qmax(R).$ In \cite{Lia_symmetric}, the symmetric version of one-sided perfect rings of quotients is defined (see \cite[Theorem 4.1]{Lia_symmetric}). In \cite{Lia_symmetric} also, the total symmetric ring of quotients $\Qtot$ is defined as a symmetric version of the total one-sided rings of quotients $\Qrtot$ and $\Qltot$ (see \cite[Theorem 5.1]{Lia_symmetric}).

\section{Right strongly semihereditary rings}
\label{section_strongly_semihereditary}

In \cite{Goodearl} (see also \cite[Theorem 7.1, XII]{Stenstrom}), Goodearl shows that for a right nonsingular
ring $R$, the following are equivalent:
\begin{itemize}
\item[(i)] Every finitely generated nonsingular module can be
embedded in a free module.

\item[(ii)] $\Qrmax(R)$ is a perfect left ring of quotients of $R.$
\end{itemize}

Using Goodearl's result, the following holds (see  \cite[Corollary 7.4, XII]{Stenstrom}). For a ring $R,$ the following conditions are equivalent:
\begin{enumerate}
\item  $R$ is right semihereditary and $\Qrmax(R)$ is a perfect left ring of quotients of $R.$

\item $R$ is right nonsingular and every finitely generated nonsingular module is projective.
\end{enumerate}
Also, if these conditions are satisfied then $R$ is also left
semihereditary and $\Qrmax(R)$ is a perfect right ring of quotients of $R.$

In \cite[Theorem 2.4, p. 54]{Evans}, Evans shows that the following conditions are
equivalent:
\begin{itemize}
\item[(3)] $R$ is right semihereditary and
$\Qrmax(R)$ is a perfect left and a perfect right ring of quotients of $R.$

\item[(4)] $R$ is right nonsingular and complemented right ideals of the matrix ring $M_n(R)$ are generated by an idempotent for all $n.$
\end{itemize}

Evans calls the rings satisfying these equivalent conditions
right strongly extended semihereditary. For brevity, we call them {\em right strongly semihereditary}. We say that a ring is {\em strongly semihereditary} if it is both left and right strongly semihereditary.

Note that a right strongly semihereditary ring is both left and right semihereditary ring ($R$ is left semihereditary by \cite[Corollary 7.4, XII]{Stenstrom}).

Condition (4) can be rephrased in terms of extendability of the ring $M_n(R).$ Recall that a right $R$-module $M$ is said to be {\em extending or CS} (``complements are summands'') if every complemented submodule of $M$ is a direct summand of $M$ (equivalently, every submodule of $M$ is essential in a direct summand of $M$, see \cite[Lemma 6.41]{Lam}). Condition (4) then can be rephrased as
\begin{itemize}
\item[(5)] $R$ is right nonsingular and the matrix ring $M_n(R)$ is extending (as a right $M_n(R)$-module) for all $n.$
\end{itemize}

If $R$ satisfies (5), then $M_n(R)$ is also right nonsingular (\cite[Exercise 18.3]{Lam}). The annihilators in a right nonsingular ring are essentially closed (\cite[Lemma 7.51]{Lam}) and, thus, complemented (\cite[Proposition 6.32]{Lam}). So, the annihilators of $M_n(R)$ are direct summands since $M_n(R)$ is extending. This demonstrates that the matrix rings over a right strongly semihereditarity ring are Baer.

By \cite[11.4, 12.8, and 12.17]{Extending_book}, (5) is equivalent to:
\begin{itemize}
\item[(6)] $R$ is right nonsingular and $R^n$ is extending (as a right $R$-module) for all $n.$

\item[(7)] $R$ is right semihereditary and $R^2$ is extending as a right $R$-module.

\item[(8)] $R$ is right nonsingular and every finitely generated projective module is extending.
\end{itemize}

Further, these conditions are equivalent with the following.
\begin{itemize}
\item[(9)]  $R$ is right nonsingular and every finitely generated nonsingular module is extending.
\end{itemize}
Indeed, (9) clearly implies (6). Conversely, assume (2) and let $M$ be a finitely generated nonsingular module and let $K$ be a submodule of $M.$ The quotient $M/\cl_{\T}(K)$ is finitely generated and nonsingular also. Thus, it is projective by (2). Hence, $\cl_{\T}(K)$ is a direct summand of $M$. Since $M$ is nonsingular, every submodule of $M$ is essential in its Goldie closure (see \cite[Proposition 7.44]{Lam}). Thus, $K$ is essential in a direct summand of $M$ and so $M$ is extending. This proves (9).

We note another condition equivalent to (1)--(9).
\begin{itemize}
\item[(10)] $R$ is right semihereditary, $\Qrmax(R)$ is a perfect left ring of quotients of $R$ and $\Qrmax(R)=\Qrtot(R).$
\end{itemize}
(10) trivially implies (3). For the converse note the following fact from \cite{Lia_quotients}: If $R$ is right semihereditary, the total right ring of quotients $\Qrtot(R)$ can be obtained as the ring of quotients with respect to the torsion theory of tensoring by $\Qrmax(R)$ (see Example \ref{perfect_example} and \cite[Theorem 12]{Lia_quotients}). This torsion theory is perfect by \cite[Lemma 8, part 3]{Lia_quotients}. Thus, if $\Qrmax(R)$ is already a perfect right ring of quotients, then $\Qrmax(R)$ is equal to $\Qrtot(R)$ (also follows from \cite[Lemma 9, part 6]{Lia_quotients}).

Thus, we have the following.
\begin{proposition} A ring $R$ is right strongly semihereditary if any of the equivalent conditions (1)--(10) holds.
In that case,  $M_n(R)$ is right strongly semihereditary as well.
\label{right_strongly_semihereditary}
\end{proposition}
\begin{proof}
We demonstrated the equivalence of conditions (1)--(10). The matrix rings over $R$ are right semihereditary by \cite[Corollary 18.6]{Lam}. Since the rings $M_m(M_n(R))$ and $M_{mn}(R)$ are isomorphic for any $m$ and $n,$ if $R$ satisfies conditions (4) or (5), all the matrix rings over $R$ satisfy those conditions as well. Thus, the matrix rings are right strongly semihereditary as well.
\end{proof}

\begin{example}
\begin{enumerate}
\item Every regular and self-injective ring is strongly semihereditary (such a ring is both right and left semihereditary and $R=\Qrmax(R)=\Qlmax(R),$ thus $R=\Qrtot(R)=\Qltot(R)$ as well). In particular, every semisimple ring is strongly semihereditary.

\item Every commutative semihereditary and noetherian ring is strongly semihereditary since for such a ring $\Qrmax(R)=\Qlmax(R)=\Qrcl(R)=\Qlcl(R),$ thus $\Qrtot(R)=\Qltot(R)=\Qrmax(R)$.

\item Let $E$ be a directed graph and $K$ a field. If $E$ is finite (i.e. has finitely many vertices and edges), the Leavitt path algebra $L_K(E)$ (introduced in \cite{Gene_Gonzalo1} and \cite{Ara_Moreno_Pardo}) is hereditary by \cite[Theorem 3.5]{Ara_Moreno_Pardo} and, thus, semihereditary as well. The condition that $E$ is a no-exit graph (i.e. no cycle in $E$ has an exit) is equivalent to the condition that $L_K(E)$ is noetherian (by \cite[Theorems 3.8 and 3.10]{AAS2}). In \cite[Proposition 5.1]{Gonzalo_Lia_noetherian}, it is shown that if $K$ is a field with a positive definite involution (i.e. $\sum_{i=1}^n \overline{k_i}k_i=0$ implies $k_i=0$ for all $i=1,\ldots,n$ and for all $n$) and $E$ a finite no-exit graph, then $\Qrmax(L_K(E))=\Qlmax(L_K(E))=\Qltot(L_K(E))=\Qrtot(L_K(E)).$ Thus, $L_K(E)$ is (left and right) strongly semihereditary. Note that $L_K(E)$ is not regular if $E$ has cycles and not commutative if $E$ has at least two connected vertices.

\item Finite $AW^*$-algebras (in particular finite von Neumann algebras) are strongly semihereditary. More generally, Baer $*$-rings studied in \cite{Lia_Baer} are strongly semihereditary. This follows from \cite[Proposition 3 and Corollary 5]{Lia_Baer}. In contrast with the previous example, this class of rings contains rings that are neither right noetherian nor right hereditary (e.g. group von Neumann algebras of infinite groups by \cite[Corollary 5]{Lia_semisimple} and \cite[Example 9.11]{Lueck}).

These examples illustrate that the class of right strongly semihereditary rings is quite versatile.

\item Let $R=\{(a_n)\in\Qset\times\Qset\times\ldots\;|\; (a_n)\mbox{ is eventually constant }\}.$ The ring $R$ is commutative so the left and the right ring of quotients coincide. The ring $R$ is regular, so $R$ is semihereditary and $\Qltot(R)=\Qrtot(R)=R.$ We also have $\Qlmax(R)=\Qrmax(R)=\Qset\times\Qset\times\ldots$ (\cite[Exercise 23, p. 328]{Lam}). Thus, $R$ is an example of a semihereditary ring that is not strongly semihereditary.
\end{enumerate}
\label{right_strongly_semihereditary_example}
\end{example}

A right strongly semihereditary ring $R$ has another favorable property -- its finitely generated projective modules match those of $\Qrmax(R).$ The following proposition has been shown for the class of Baer $*$-rings from \cite{Lia_Baer} and for noetherian Leavitt path algebras in \cite{Gonzalo_Lia_noetherian}. We note that it holds for every right strongly semihereditary ring.

\begin{proposition} Let $R$ be right strongly semihereditary. Let $Q$ denote $\Qrmax(R).$
\begin{enumerate}
\item[(i)] For every finitely generated nonsingular (equivalently projective) $R$-module $P$, there is
a one-to-one correspondence \[\{\mbox{direct summands of }P\}
\longleftrightarrow\{\mbox{direct summands of }E(P)=P\otimes_{R}Q\}\] given by
$N \mapsto$ $N\otimes_{R}Q = E(N).$ The inverse map is given by
$L\mapsto L\cap P.$

\item[(ii)] The isomorphism $\varphi: \ve(R)\rightarrow \ve(Q)$ of monoids of isomorphism classes of finitely generated projective modules induced by the map  $P\mapsto P\otimes_{R}Q$ has the inverse induced by $S\mapsto S\cap R^n$ if $S$ is a finitely generated projective $Q$-module that can be embedded in $Q^n.$
\end{enumerate}
\label{K0_theorem}
\end{proposition}
\begin{proof}
If $P$ is finitely presented nonsingular, then \[P\otimes_R Q= E(P)\] by \cite[12.14]{Extending_book}. As a consequence, there is an embedding of monoids of isomorphism classes of finitely generated projective modules $\ve(R)\rightarrow \ve(Q)$ induced by the map $P\mapsto P\otimes_R Q.$ The claim (i) then follows from the one-to-one correspondence between complements in $P$ and direct summands of $E(P)$ expanded on in Example \ref{Goldie_example} and the fact that complements in $P$ are exactly direct summands of $P$ by condition (8). The claim (ii) follows directly from (i).
\end{proof}

We take a closer look at closures in different torsion theories now.

\begin{proposition} Let $R$ be any ring and $P$ be a submodule of $R^n.$
\begin{enumerate}
\item \[\begin{array}{rcl}
\cl_{\bnd}(P) & = & \{x\in R^n | f(x)=0 \mbox{ if }f\in\homo_R(R^n,R)\mbox{  with }P\subseteq \ker f\}\\
 & = &\bigcap\{\ker f | f\in\homo_R(R^n,R)\mbox{ with }P\subseteq \ker f\}\\
 & \subseteq &\bigcap\{ S | S\mbox{ is a direct summand of }R^n\mbox{ with }P\subseteq S\}\\
\end{array}\]

If $R$ is right semihereditary, then the last inclusion is an equality.

\item If $R$ is right nonsingular,
\begin{flushleft}
\begin{tabular}{rcl}
$\cl_{\T}(P)$ & = & largest submodule of $R^n$ in which $P$ is essential\\
& = & smallest complement containing $P$ in $R^n$ \\
& = & intersection of complements containing $P$ in $R^n.$
\end{tabular}\end{flushleft}

\item If $R$ is right semihereditary such that $M_n(R)$ is Baer, then $\cl_{\bnd}(P)$ is a direct summand of $R^n.$

\item If $R$ is right strongly semihereditary, $\cl_{\T}(P)=E(P)\cap R^n=\cl_{\bnd}(P)$ and it is a direct summand of $R^n.$
\end{enumerate}
\label{VariousClosures}
\end{proposition}
\begin{proof}
(1) The first two sets are equal by the definition of closure with respect to the torsion theory $(\bnd, \unb).$ The second two sets are clearly equal.

If $S$ is a direct summand of $R^n,$ then $S$ is a direct summand of $\cl_{\bnd}(S)$ too, so the torsion quotient $\cl_{\bnd}(S)/S$ embeds in the torsion-free module $R^n.$ Thus, $\cl_{\bnd}(S)/S$ is trivial and so $\cl_{\bnd}(S)=S.$ This demonstrates that if $S$ is a direct summand of $R^n$ that contains $P,$ then $\cl_{\bnd}(P)\subseteq \cl_{\bnd}(S)=S.$ Thus, the first set is contained in the fourth.

If $R$ is right semihereditary, every  $f\in\homo_R(R^n,R)$ with $P\subseteq \ker f$ has the image that is finitely generated and projective and so $\ker f$ is a direct summand of $R^n$. Thus, every such map $f$ determines one direct summand $S$ of $R^n$ with $P\subseteq S$ and so the fourth set is contained in the third.

(2) If $R$ is right nonsingular, the equality of the four sets follows from \cite[Proposition 7.44]{Lam}. The module $P$ is essential in closure $\cl_{\T}(P)$ since $R^n$ is nonsingular.

(3) If $M_n(R)$ is Baer, the completeness of the lattice of idempotents of $M_n(R)$ implies the completeness of the lattice of direct summands of $R^n.$ Thus the claim follows from part (1).

(4) If $R$ is right strongly semihereditary, every complement in $R^n$ is a direct summand. So, the intersection of all complements of $R^n$ containing $P$ is the same as the intersection of all direct summands of $R^n$ containing $P.$ Thus, $\cl_{\T}(P)=\cl_{\bnd}(P)$ and it is a direct summand by part (3).

The equality $\cl_{\T}(P)=E(\cl_{\T}(P))\cap R^n$ holds by one-to-one correspondence from Example \ref{Goldie_example}. The module $P$ is an essential subset of $\cl_{\T}(P)$ since $R^n$ is nonsingular and so $E(P)=E(\cl_{\T}(P)).$ Thus $\cl_{\T}(P)=E(P)\cap R^n.$
\end{proof}

The following theorem has also been shown for the class of Baer $*$-rings studied in \cite{Lia_Baer}. We show that it holds for all right strongly semihereditary rings.

\begin{theorem} Let $R$ be a right semihereditary ring such that $M_n(R)$ is Baer for every $n$, $M$ a finitely generated $R$-module and $K$ a submodule of $M.$
\begin{enumerate}
\item The module $M/\cl_{\bnd}(K)$ is finitely generated projective and
$\cl_{\bnd}(K)$ is a direct summand of $M.$ In particular, $M$ splits to a direct sum of finitely generated projective module $\unb M$ and torsion module $\bnd M$
\[M=\bnd M\oplus \unb M.\]

\item If $R$ is right strongly semihereditary, then $\cl_{\T}(K)=\cl_{\bnd}(K)$ is a direct summand of $M$ and $(\bnd, \unb)=(\T, \bigP)$ on the class of finitely generated modules.
\end{enumerate}
\label{SplittingOfBnd}
\end{theorem}
\begin{proof}
(1) If $M$ is $R^n$, the claim follows by Proposition \ref{VariousClosures}.

Now let $M$ be any finitely generated module. There is a
nonnegative integer $n$ and an epimorphism $f: R^n\rightarrow M.$
Note that $\cl_{\bnd}(f^{-1}(K)) = f^{-1}(\cl_{\bnd}(K)).$ The proof of this fact
can be obtained directly following definitions. More details can also be found in the proof of \cite[Theorem 11]{Lia_Baer}.

It is easy to see that $f: R^n\rightarrow M$ induces an
isomorphism of $R^n/f^{-1}(\cl_{\bnd}(K))$ onto
$M/\cl_{\bnd}(K).$ Since $ \cl_{\bnd}(f^{-1}(K))
=f^{-1}(\cl_{\bnd}(K))$ and $R^n/\cl_{\bnd}(f^{-1}(K))$ is
finitely generated projective, $M/\cl_{\bnd}(K)$ is finitely generated projective as well.
So $0\rightarrow\cl_{\bnd}(K)\rightarrow M\rightarrow
M/\cl_{\bnd}(K)\rightarrow 0 $ splits.

(2) Since $(\T, \bigP)\leq (\bnd, \unb),$ $\cl_{\T}(K)\subseteq\cl_{\bnd}(K).$ The module $\cl_{\bnd}(K)/\cl_{\T}(K)$ is a quotient of the bounded module $\cl_{\bnd}(K)/K,$ so it is bounded itself. On the other hand, if $R$ is right strongly semihereditary, $M/\cl_{\T}(K)$ is finitely generated projective as it is a finitely generated nonsingular module. The module $M/\cl_{\bnd}(K)$ is finitely generated projective by part (1). Thus the short exact sequence
$$\xymatrix{ 0 \ar[r] & \cl_{\bnd}(K)/\cl_{\T}(K) \ar[r] & M/\cl_{\T}(K) \ar[r] & M/\cl_{\bnd}(K) \ar[r] & 0 }$$
splits and so $\cl_{\bnd}(K)/\cl_{\T}(K)$ is finitely generated projective also. But every finitely generated projective module is unbounded. So, the quotient $\cl_{\bnd}(K)/\cl_{\T}(K)$ is both bounded and unbounded and hence it has to be zero. Thus $\cl_{\bnd}(K)=\cl_{\T}(K).$
\end{proof}

If $R$ is a right strongly semihereditary ring, then every finitely generated nonsingular right module is extending. However, this statement does not hold if the assumption on nonsingularity of the module is dropped (in particular, condition (9) cannot be weakened by dropping the assumption that the module is nonsingular). This is because a ring whose all finitely generated modules are extending is necessarily of finite uniform dimension (by \cite{Osofsky_Smith}, see also \cite[Corollary 6.45]{Lam}) and right noetherian (by \cite[Theorem 5]{Huynh_Rizvi_Yousif}). There are examples of finite von Neumann algebras that have infinite uniform dimension and that are not right noetherian (in fact, all group von Neumann algebras of infinite groups are not noetherian and do not have finite uniform dimension by \cite[Corollary 5]{Lia_semisimple} and \cite[Example 9.11]{Lueck}).

The same examples can be used to demonstrate that the torsion theories  $(\T,\bigP)$ and $(\bnd, \unb)$ are different in general. Thus,
although $\T=\bnd$ on the class of finitely generated modules of a right strongly semihereditary ring, these two classes are different in general (in particular, see \cite[Exercise 6.5]{Lueck}).

Also, although $\T=\smallt$ on the class of finitely presented modules of a right strongly semihereditary ring (which can be shown following the argument of the proof of \cite[Proposition 21]{Lia_Baer} once the dimension that we introduce in the remainder of the paper becomes available), these two classes are different in general (see \cite[Example 8.34]{Lueck}).

\section{Involutive strongly semihereditary rings}
\label{section_involutive_strongly_semihereditary}

For rings with involution, the concept of strong semihereditarity is left-right symmetric and has further favorable properties. Before we can prove them, we recall a few facts about symmetric rings of quotients of involutive rings from \cite{Gonzalo_Lia_noetherian}. A left Gabriel filter $\ef_l$ and a right Gabriel filter $\ef_r$ are said to be {\em conjugated} if and only if $\ef_l^*=\ef_r$ (i.e. $I\in \ef_r$ if and only if $I^*=\{r^*\ |\ r\in I\}\in \ef_l).$ To shorten our notation, we write $\Qrmax$ for $\Qrmax(R),$ $\Qltot$ for $\Qltot(R)$ and similar abbreviations are used for other rings of quotients if it is clear that they are quotients of $R.$

\begin{proposition}\cite{Gonzalo_Lia_noetherian}
Let $R$ be an involutive ring.
\begin{itemize}
\item[(a)] If $\ef_l$ and $\ef_r$ are conjugated left and right Gabriel filters and if the involution extends to $R_{\ef_r}$ then $R_{\ef_r}$ is also a left ring of quotients and $_{\ef_l}R=R_{\ef_r}.$ Similarly, if it extends to $_{\ef_l}R,$ then $_{\ef_l}R$ is a right ring of quotients and $_{\ef_l}R=R_{\ef_r}.$

\item[(b)] If $\ef_l$ and $\ef_r$ are conjugated left and right Gabriel filters, then the involution extends to the symmetric ring of quotients $_lR_r.$ If $R$ is $\tau_r$-torsion-free ($\tau_l$-torsion-free) and the involution extends to $R_{\ef_r}$ ($_{\ef_l}R$)
then $R_{\ef_r}=\,_{\ef_l}R=\, _lR_r.$

\item[(c)] If $S$ is a perfect symmetric ring of quotients with an injective localization map $q:R\rightarrow S$, then the involution extends to $S$ making $q$ a $*$-homomorphism and the left filter  $\ef_l=\{I | Sq(I)=S\}$ is conjugated to the right filter $\ef_r=\{J | q(J)S=S\}.$

\item[(d)] The involution extends to $\Qmax$ and $\Qtot.$

\item[(e)]  If the involution extends to $\Qrmax$ or $\Qlmax,$ then $\Qrmax=\Qlmax=\Qmax.$

\item[(f)] If the involution extends to $\Qrtot$ or $\Qltot,$ then $\Qrtot=\Qltot=\Qtot.$
\end{itemize}
\label{quotients_of_involutive_ring}
\end{proposition}
The proof of all these claims is contained in the proofs of \cite[Proposition 4.1, 4.2, Lemma 4.3, Corollary 4.4]{Gonzalo_Lia_noetherian}.

Now we prove a characterization theorem for involutive strongly semihereditary rings.

\begin{proposition} If $R$ is a ring with involution, the following are equivalent.
\begin{itemize}
\item[(1)] $R$ is right strongly semihereditary.

\item[(2)] $R$ is left strongly semihereditary.

\item[(3) (3$'$)] $R$ is right (left) semihereditary and $\Qrmax=\Qtot$ ($\Qlmax=\Qtot$).

\item[(4) (4$'$)] $R$ is right (left) semihereditary and the involution can be extended to $\Qrmax=\Qrtot$ ($\Qlmax=\Qltot$).

\item[(5) (5$'$)] $R$ is right (left) semihereditary and $\Qrmax=\Qltot=\Qrtot$ ($\Qlmax=\Qltot=\Qrtot$).

\item[(6) (6$'$)] $R$ is right (left) semihereditary and $\Qrmax=\Qltot$ ($\Qlmax=\Qrtot$).

\item[(7) (7$'$)] $M_n(R)$ is Baer for every $n$ and the involution can be extended to $\Qrmax$ ($\Qlmax$).

\item[(8)] $M_n(R)$ is Baer and $\Qrmax(M_n(R))=\Qlmax(M_n(R))$ for every $n.$
\end{itemize}
Under any of these conditions, the ring $M_n(R)$ is strongly semihereditary for any $n$ and the six rings of quotients $\Qrmax, \Qlmax, \Qmax, \Qrtot, \Qltot,$ and $\Qtot$ of the ring $M_n(R)$ are equal.
\label{strongly_semihereditary_involutive_rings}
\end{proposition}
\begin{proof} First, note that an involutive right semihereditary ring is also left semihereditary.
If $R$ has involution, this involution extends to all matrix rings $M_n(R)$ by $(a_{ij})^*=(a_{ji}^*).$ So, the complemented right ideals are direct summands iff the complemented left ideals are. Thus (1) and (2) are equivalent. Also, the equivalence of any of conditions (3)--(7) with their corresponding ``prime'' conditions (3$'$)--(7$'$) follows by left-right symmetry.

(1) implies that $\Qrmax$ is left and right perfect ring of quotients. Thus, $\Qrmax$ is contained in $\Qtot.$ Since the converse always holds, (1) implies (3).

(3) implies that $\Qrmax$ is a symmetric ring of quotients. Thus, the involution extends to it by parts (b) and (c) of Proposition \ref{quotients_of_involutive_ring}. Also $\Qrmax=\Qtot$ implies that $\Qrmax=\Qrtot$ also since it is always the case that $\Qtot\subseteq\Qrtot\subseteq\Qrmax.$ So (3) implies (4).

(4) implies that the involution extends to $\Qrtot.$ Thus, $\Qrtot=\Qltot$ by part (f) of Proposition \ref{quotients_of_involutive_ring}. So (4) implies (5).

(5) clearly implies (6).

(6) implies condition (1) of Proposition \ref{right_strongly_semihereditary}, so condition (1) of this proposition follows too.

This shows that (1)--(6) are all equivalent. Now we show that (1)--(6) $\Rightarrow$ (7) $\Rightarrow$ (8) $\Rightarrow$ (1).

(1) implies that $M_n(R)$ is Baer, as it was observed in the previous section. By (4), the involution extends to $\Qrmax.$ Thus (1)--(6) implies (7).

If the involution can be extended to $\Qrmax,$ then the involution can be extended to $M_n(\Qrmax).$ Since $\Qrmax(M_n(R))\cong M_n(\Qrmax(R))$ (\cite[Exercise 17.16]{Lam}), this implies that the involution extends to $\Qrmax(M_n(R))$ and so   $\Qrmax(M_n(R))=\Qlmax(M_n(R))$ by part (e) of Proposition \ref{quotients_of_involutive_ring}. Thus (7) implies (8).

(8) implies that $M_n(R)$ is a Baer ring with equal left and right maximal ring of quotients. By \cite[Corollary 12.7]{Extending_book}, $M_n(R)$ is left and right nonsingular and left and right extending. Thus, condition (5) of Proposition \ref{right_strongly_semihereditary} is satisfied. So, (1) holds.

Finally, note that $M_n(R)$ is right strongly semihereditary (by Proposition \ref{right_strongly_semihereditary}) involutive ring. The equality of six quotients then follows from conditions (5) and (8).
\end{proof}

We prove further properties of involutive strongly semihereditary rings.

\begin{proposition} If $R$ is an involutive strongly semihereditary ring with $Q=\Qrmax(R)$, then the following hold.
\begin{enumerate}

\item The extension of the involution to $Q$ is unique (it exists by Proposition \ref{strongly_semihereditary_involutive_rings}).

\item $Q$ is unit-regular and directly finite (i.e. $xy=1$ implies that $yx=1$ for all $x,y$). Thus $R$ is directly finite and $Q$ and $R$ are finite (i.e. $xx^*=1$ implies that $x^*x=1$ for all $x$).

\item $Q$ is strongly semihereditary.

\item The following conditions are equivalent:
\begin{itemize}
\item[(i)]  $R$ is a Baer $*$-ring,
\item[(ii)] $R$ is a Rickart $*$-ring,
\item[(iii)] $Q$ and $R$ have the same projections.
\end{itemize}
Note that (i)--(iii) hold if $R$ is symmetric (i.e. $1+xx^*$ is invertible for every $x\in R$).

\item If the involution on $R$ is $n$-proper (i.e. $\sum_{i=1}^n x_i^*x_i=0$ implies $x_i=0$ for all $i=1,\ldots,n$),
then its extension on $Q$ is $n$-proper as well. Thus, if the involution on $R$ is positive definite (i.e. $n$-proper for every $n$), then its extension on $Q$ is positive definite as well.

\item If the involution on $R$ is $n$-proper, then $M_n(Q)$ is a $*$-regular Baer $*$-ring. Any of the three equivalent conditions from (4) for $M_n(R)$ implies that the involution on $R$ is $n$-proper.
\end{enumerate}
\label{properties}
\end{proposition}

\begin{proof}
(1) The uniqueness follows from the proof of \cite[Corollary 3.4]{Gonzalo_Lia_noetherian}.  \cite[Corollary 3.4]{Gonzalo_Lia_noetherian} is formulated for Leavitt path algebras but its proof holds for any involutive strongly semihereditary ring.

(2) Since $Q=\Qrmax=\Qlmax,$ $Q$ is both left and right self-injective. Every regular left and right self-injective ring is unit-regular (see \cite[Theorem 9.29]{Goodearl_book}). Every unit-regular ring is directly finite (see \cite[Proposition 5.2]{Goodearl_book}). The ring $R$ is directly finite since $Q$ is. Note that the direct finiteness of $R$ also follows from \cite[Theorem 6.48]{Lam}.

(3) Since $\Qrmax(\Qrmax(R))=\Qrmax(R)$ and $\Qrtot(\Qrtot(R))=\Qrtot(R),$ then $\Qrtot(Q)=\Qrtot(\Qrtot(R))=\Qrtot(R)=Q=\Qrmax(\Qrmax(R))=\Qrmax(Q).$ Similarly, $\Qlmax(Q)=\Qltot(Q)=Q.$ Thus, $Q$ is strongly semihereditary.

(4) (i) and (ii) are equivalent since $R$ is Baer (thus it is a Baer $*$-ring if and only if it is a Rickart $*$-ring, see \cite[Proposition 1.24]{Berberian_web}). If the involution of a Baer $*$-ring $R$ can be extended to its maximal right ring of quotients $Q,$ then $Q$ and $R$ have the same projections (see \cite[Proposition 21.2]{Berberian_web}, also \cite{Pyle}) so (ii) implies (iii). Conversely, let us prove (i) assuming (iii). To prove that $R$ is a Rickart $*$-ring it is sufficient to show that every right principal ideal of $R$ generated by an idempotent is also generated by a projection. Let $e$ be an idempotent of $R$. Then $eQ=pQ$ for some projection $p$ in $Q$ by $*$-regularity of $Q.$ By assumption, $p$ is in $R$. Then $eR=eQ\cap R=pQ\cap R=pR.$ This proves (i).

A Baer ring that is symmetric is a Baer $*$-ring by \cite[Exercises 5A p.25]{Berberian}. So, if $R$ is symmetric, (i)--(iii) hold.

(5) Although it is formulated for Leavitt path algebras, the proof of \cite[Proposition 3.1]{Gonzalo_Lia_noetherian} demonstrates that the extension of an $n$-proper involution to an overring in which the ring is dense is $n$-proper as well. Thus, the claim holds for $R$ and $Q.$

(6) By (5), the involution on $Q$ is $n$-proper. Thus the $*$-transpose involution on $M_n(Q)$ is proper (see \cite[Lemma 2.1]{Gonzalo_Ranga_Lia}). The ring $M_n(Q)$ is $*$-regular since a regular $*$-ring with proper involution is $*$-regular (see \cite[Exercise 6A, \S 3]{Berberian}). A Baer and $*$-regular ring is a Baer $*$-ring (see \cite[Exercise 4A, p. 25]{Berberian}) so $M_n(Q)$ is a Baer $*$-ring as well. Any of the three equivalent conditions from (4) for $M_n(R)$ implies that the involution on $M_n(R)$ is proper (thus the involution on $R$ is $n$-proper) since an involution of a Rickart $*$-ring is necessarily proper.
\end{proof}

\begin{example}
\begin{enumerate}
\item Finite $AW^*$-algebras (in particular finite von Neumann algebras) are strongly semihereditary rings with proper involution. Since these algebras are Baer $*$-rings, their maximal rings of quotients do not have any new projections.

\item Recall that Leavitt path algebras are involutive rings (for any involution $k\mapsto\overline{k}$ on the field $K$, we have the involution $\sum k pq^*\mapsto \sum \overline{k}qp^*$ on $L_K(E)$). Thus, by part (3) of Example \ref{right_strongly_semihereditary_example}, noetherian Leavitt path algebras over positive definite fields are strongly semihereditary involutive rings. If the involution on $K$ is positive definite, the involution on $L_K(E)$ is positive definite too (by \cite[Proposition 2.4]{Gonzalo_Ranga_Lia}).

\item Although a noetherian Leavitt path algebra $L_K(E)$ is a Baer ring with involution (positive definite if the involution on $K$ is positive definite), it is interesting to note that $L_K(E)$ does not have to be a Baer $*$-ring (nor a Rickart $*$-ring). Thus, not all algebraic and $*$-concepts agree on $L_K(E).$ The regular ring $Q_K(E),$ on the other hand, is a Baer $*$-ring if the involution on $K$ is positive definite (by part (6) of Proposition \ref{properties}).

The author is grateful to Pere Ara for inspiring conversations during the conference X Jornadas de Teor\'ia de Anillos at the Universitat Aut\`onoma de Barcelona that has lead to this example.

Let $R$ be the $2\times 2$ matrix algebra $M_2(K[x,x^{-1}])$ over Laurent polynomial ring for any positive definite field $K.$ This algebra can be realized as the Leavitt path algebra over $K$ of any of the following no-exit graphs.
$$
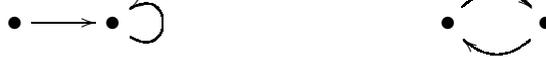
\xymatrix{ {\bullet} \ar[r] & {\bullet}
\ar@(dr,ur)} \hskip4cm \xymatrix{{\bullet} \ar@/^1pc/ [r]   & {\bullet} \ar@/^1pc/ [l] }$$
Consider the element \[e=\left[\begin{array}{cc} 1 & 1+x\\ 0 & 0 \end{array}\right].\]
We claim that $e$ is an idempotent such that $eR$ is strictly larger than $ee^*R.$ It is straightforward to check that $e$ is an idempotent and that the element $e_{11}=\left[\begin{array}{cc} 1 & 0\\0 & 0 \end{array}\right]$ is in $eR.$ Assume that $e_{11}$ is in $ee^*R.$ By equating the first-column-first-row elements we obtain that $1+(1+x)(1+x^{-1})=3+x+x^{-1}$ is invertible which is a contradiction. Thus  $e_{11}\notin ee^*R.$

An involutive ring $R$ is Rickart $*$ iff it is Rickart and $eR= ee^*R$ for all idempotents $e$ (see \cite[Proposition 1.11]{Berberian_web}). Thus, the fact that $eR\neq ee^*R$ implies that the Rickart ring $R$ is not a Rickart $*$-ring. Note that this also implies that there are more projections in the regular ring $Q=Q_K(E)$ than in $R$ (by Proposition \ref{properties}). Thus, we have a situation that is not present for finite $AW^*$-algebras.

\item Note that $R$ from part (5) of Example \ref{right_strongly_semihereditary_example} is an involutive ring (since it is commutative, the identity map is an involution). Thus, this also illustrates that not all semihereditary involutive rings are strongly semihereditary.
\end{enumerate}
\label{examples_strongly_semihereditary}
\end{example}

These examples show that the class of involutive strongly semihereditary contains quite diverse rings: some of them are Baer $*$ (resp. noetherian, hereditary) while the others are not. This further motivates the interest in this class.

Using Propositions \ref{right_strongly_semihereditary} and \ref{strongly_semihereditary_involutive_rings}, we obtain a characterization of extending Leavitt path algebras. It may not be a surprise that finite no-exit graphs yield extending Leavitt path algebras, but it is interesting to note that extending Leavitt path algebras arise {\em only} from finite no-exit graphs. Also, these graphs are the only graphs with Leavitt path algebras that are strongly semihereditary.

\begin{corollary} Let $E$ be a finite graph and $K$ a field with a positive definite involution. 
The following conditions are equivalent.
\begin{itemize}
\item[(i)] $L_K(E)$ is extending (as a right $L_K(E)$-module).

\item[(ii)] $E$ is a no-exit graph.

\item[(iii)] $L_K(E)$ is strongly semihereditary.

\item[(iv)] $M_n(L_K(E))$ is extending (as a right $M_n(L_K(E))$-module) for every $n$.
\end{itemize}
\label{extending_LPAs}
\end{corollary}
\begin{proof}
(i) $\Rightarrow$ (ii). Since $L_K(E)$ is an involutive ring, if $L_K(E)$ is extending as a right $L_K(E)$-module, then it is extending as a left $L_K(E)$-module as well. Thus, it is directly finite by \cite[Theorem 6.48]{Lam}. Hence, $E$ is a no-exit graph by \cite[Theorem 3.10]{Gonzalo_Lia_noetherian}.

(ii) $\Rightarrow$ (iii). Assume that $E$ is a no-exit graph. Then the involution extends to the maximal left ring of quotients by \cite[Theorem 3.10]{Gonzalo_Lia_noetherian}. Since the algebra $M_n(L_K(E))$ is also a Leavitt path algebra of a no-exit graph $M_nE$ (the graph $M_nE$ is obtained from $E$ by adding the oriented line of length $n-1$ to every vertex of $E$, see \cite[Definition 9.1 and Proposition 9.3]{Abrams_Tomforde}), $M_n(L_K(E))$ is Baer by \cite[Corollary 3.14]{Gonzalo_Lia_noetherian}. Thus, we have that condition (7$'$) of Proposition \ref{strongly_semihereditary_involutive_rings} holds and so $L_K(E)$ is strongly semihereditary.

(iii) implies (iv) by Proposition \ref{right_strongly_semihereditary} and (iv) trivially implies (i).
\end{proof}

\section{Involutive rings with dimension}
\label{section_dimension_general}

In this section, we develop general theory of dimension defined for a general equivalence of projections $\sim$. We later use this general theory to obtain dimension of specific classes of rings and for specific equivalences of projections.

Recall that the algebraic equivalence $\sima$ of idempotents (see Section \ref{section_introduction}, for more details see  \cite[Section 5]{Berberian_web}) and the order $\leq$ of idempotents (defined by $e\leq f$ iff $ef=fe=e$ iff $e\in fRf$) give rise to another relation
\[e\manjea f \;\;\;\mbox{ iff }\;\;\; e\sima f'\mbox{ for some }f'\leq f.\]
We refer to this relation as {\em algebraic domination}.

In any involutive ring, however, these concepts naturally adapt to projections: the $*$-equivalence $\sims$ (see Section \ref{section_introduction}, for more details see \cite[Section 6]{Berberian_web}) of projections and the order $\leq$ of projections, give rise to the dominance relation defined by
\[p\manjes q\;\;\; \mbox{ iff }\;\;\; p\sims q'\mbox{ for some }q'\leq q.\]
We refer to this relation as {\em $*$-domination}.

Thus, in a $*$-ring we have at least two different equivalence and dominance relations of interest. In \cite{Berberian_web}, Berberian unifies both algebraic and $*$-concepts by considering a general equivalence relation $\sim$ giving rise to general dominance relation $\preceq.$ By studying these concepts, Berberian generalizes his work from \cite{Berberian} where just $*$-equivalence was considered. Following this idea, we start by considering the general setting.

One of the benefits of the generalized approach is that we directly obtain most of the results from \cite{Lia_Baer} simply by interpreting a general equivalence $\sim$ as $*$-equivalence on Baer $*$-ring satisfying nine axioms considered in \cite{Lia_Baer}. We also obtain some  results on involutive strongly semihereditary rings and noetherian Leavitt path algebras in particular, by interpreting general equivalence $\sim$ as algebraic equivalence $\sima.$

We proceed as follows. First, we recall general axioms on Baer $*$-ring related to a general equivalence $\sim.$ These axioms have been referred to as {\em Kaplansky's axioms} in \cite{Berberian_web}. Second, we consider a general dimension function on projections in a Baer $*$-ring. This dimension corresponds to the dimension function as in \cite[Chapter 6]{Berberian} but is not related specifically to $*$-equivalence. We show how it can be extended to dimension of all modules as it was done for finite von Neumann algebras in \cite{Lueck_dimension}. We intent to use these results in the next section in order to apply this general theory to involutive strongly semihereditary rings that do not even have to be Baer $*$-rings.

We start by any equivalence $\sim$ defined on projections in a Baer $*$-ring $R$. It defines a general dominance relation by
\[p\preceq q\;\;\; \mbox{ iff }\;\;\; p\sim q'\mbox{ for some }q'\leq q\]
We consider the following axioms.

\begin{itemize}
\item[(Def)] $p\sim 0$ implies $p=0.$ In this case, $p$ is said to be {\em definite}.

\item[(CC)] $p\sim q$ implies $cp\sim cq$ for every central projection $c.$ In this case, $p$ and $q$ are {\em centrally comparable}.

\item[(IP)] If $\{p_i\}_{i\in I}$ is an orthogonal family of projections with $\sup p_i = p$, and
if $p\sim q$, then there exists an orthogonal family of projections $\{q_i\}_{i\in I}$ such that
$q=\sup q_i$ and $p_i\sim q_i$ for all $i\in I.$ In this case $\{q_i\}$ is said to be an {\em induced partition} of $\{p_i\}.$

\item[(FA)] If $p_1,\ldots, p_n$ are orthogonal projections with  $\sup p_i = p$, and $q_1,\ldots,q_n$
orthogonal projections with $\sup q_i=q,$ $p_i\sim q_i,$ for all $i=1,\ldots, n$, then $p\sim q.$ This is known as {\em finite
additivity}.

\item[(OA)] If $\{p_i\}_{i\in I}$ is an orthogonal family of projections with $\sup p_i = p$,
$\{q_i\}_{i\in I}$ an orthogonal family of projections with $\sup q_i=q,$ $p_i\sim q_i$ for all $i\in I$ and $pq=0,$ then $p\sim q.$ This is known as {\em orthogonal additivity}.

\item[(CA)] If $\{c_i\}_{i\in I}$ is an orthogonal family of central projections with $\sup c_i = 1$, and if $p$ and $q$ are projections such that $c_ip\sim c_iq,$ for all $i\in I,$ then  $p\sim q.$ This is known as {\em central additivity}.

\item[(Add)] (OA) holds without the assumption that $pq=0.$ This is known as {\em additivity of equivalence} or {\em complete additivity}.

\item[(P)] For every two projections $p$ and $q,$ \[ p - \inf\{ p,q\}\sim
\sup\{ p,q\}-q.\] This property is referred to as the {\em parallelogram law.}

\item[(GC)] For projections $p$ and $q,$ there is a
central projection $c$ such that
\[ cp\preceq cq\;\;\mbox{ and }\;\;(1-c)q\preceq
(1-c)p.\] In this case, we say that $p$ and $q$ are {\em generalized comparable}.

\item[(Fin)] $p\sim 1$ implies $p=1.$ In this case, we say that 1 is a finite projection and that $R$ is {\em finite} ring for $\sim.$
\end{itemize}

It has been showed that (Def) to (OA) and (P) imply (GC) and (Add) (\cite[Theorem 2.1]{Maeda_Holland}).

The axioms (Def), (CC), (IP) and (FA)  hold in a Baer $*$-ring both for $\sima$ and $\sims$ and (OA) holds for $\sims$ (see \cite[p. 47]{Kaplansky}). If $R$ is regular in addition to being Baer $*$-ring, then (OA) and (P) hold for $\sima$ (see \cite[p. 47]{Kaplansky}) thus (GC) and (Add) hold for $\sima$ as well (by \cite[Theorem 2.1]{Maeda_Holland}). Also, if $R$ is a regular Baer $*$-ring (thus $*$-regular as well), then (Fin) holds for $\sims$ (see \cite[Theorem 3.1]{Ara_Menal}). If $R$ is regular and self-injective, then (Fin) holds for $\sima$ (see \cite[Prop. 5.2 and Thm. 9.29]{Goodearl_book}).

Before relaxing our assumptions in the next section, we consider Baer $*$-rings with (Fin) and (GC). This is because such rings have a dimension function on the set of all projections uniquely determined by four favorable properties and satisfying additional seven properties. Moreover, we would like to have the same dimension function on the matrix rings $M_n(R)$ for all $n,$ so we  assume that $M_n(R)$ is a Baer $*$-ring with (Fin) and (GC) for every $n$ as well. Equipped with such a dimension function, our goal is to define a dimension of finitely generated projective $R$-modules first and then to extend it to a dimension of any $R$-module. Finally, using $\Qrmax(R)$ and its favorable properties for strongly semihereditary $*$-ring $R$, we obtain a dimension even in case when $R$ is not necessary a Baer $*$-ring.

Let us first discuss the domain of a dimension function. Let $R$ be a Baer $*$-ring that satisfies (Def) to (OA), (GC) and (Fin). Let $Z$
denote the center of $R.$ The projection lattice $P(Z)$ of $Z$ is
a complete Boolean algebra and, as such, may be identified with
the Boolean algebra of closed-open subspaces of a Stonian space
$X$. The space $X$ can be viewed as the set of maximal ideals of
$P(Z);$ $p\in P(Z)$ can be identified with the closed-open subset
of $X$ that consist of all maximal ideals that exclude $p.$

The algebra $C(X)$ of continuous complex-valued functions on $X$
is a commutative $AW^*$-algebra. An element $p\in P(Z)$ can be
viewed as an element of $C(X)$ by identifying $p$ with the
characteristic function of the closed-open subset of $X$ to which
$p$ corresponds. The positive unit ball $\{p\in C(X)| 0\leq p\leq 1\}$ is a complete lattice (\cite[p. 162]{Berberian}).

\begin{theorem}\cite[Sec. 33]{Berberian}, \cite[Sec. 19]{Berberian_web}
Let $R$ be a Baer $*$-ring that satisfies (Def) to (OA), (GC) and (Fin) for a relation $\sim$. Then there exists a unique function $d: P(R)\rightarrow C(X)$ such
that (D1)--(D4) hold.
\begin{itemize}
\item[(D1)] $p\sim q$ implies $d(p)=d(q),$

\item[(D2)] $d(p)\geq 0,$

\item[(D3)] $d(c)=c$ for every $c\in P(Z),$

\item[(D4)] $pq=0$ implies $d(p+q)=d(p)+d(q).$
\end{itemize}
We call the function $d$ the {\em dimension function} for $\sim$. It
satisfies the following properties:
\begin{itemize}
\item[(D5)] $0\leq d(p)\leq 1,$

\item[(D6)] $d(cp)=cd(p)$ for every $c\in P(Z),$

\item[(D7)] $d(p)=0$ iff $p=0,$

\item[(D8)] $p\sim q$ iff $d(p)=d(q),$

\item[(D9)] $p\preceq q$ iff $d(p)\leq d(q),$

\item[(D10)] If $p_i$ is an increasingly directed family of
projections with supremum $p$, then $d(p)= \sup d(p_i),$

\item[(D11)] If $p_i$ is an orthogonal family of projections with
supremum $p$, then $d(p)= \sum d(p_i).$
\end{itemize}
\label{dim_function_exists}
\end{theorem}

\cite[Chapter 6]{Berberian} is devoted to the proof of this theorem for $\sims$. In \cite[Section 19]{Berberian_web}, Berberian claims that the proofs in \cite{Berberian} transfer to any general relation $\sim$ on a Baer $*$-ring under assumptions imposed on $R$ in the theorem above. As a corollary, the dimension function for $\sima$ exists on any regular Baer $*$-ring (\cite[19.7]{Berberian_web}).

The dimension function for an equivalence relation $\sim$ is uniquely determined by $\sim.$ Moreover, in Proposition \ref{uniqueness_of_dim} we prove a stronger statement. To prove the proposition, we need a few preliminary facts and a lemma. The preliminary facts also illustrate the idea of the proof of Theorem \ref{dim_function_exists}.

A projection $p$ is said to be {\em simple} if there is a central projection $c,$ a positive integer $n$ and orthogonal projections $p_i,$ $i=1,\ldots, n$ with $p_i\sim p$ for all $i,$ such that $c=p_1+\ldots+p_n.$ By \cite[Proposition 10, sec. 26]{Berberian}, $c$ is unique and we denote it by $C(p)$ following the notation from \cite{Berberian}. Positive integer $n$ is also unique by \cite[Proposition 1, sec. 26]{Berberian}. This integer is called the {\em order} of a simple projection $p.$ The dimension of a simple projection of order $n$ thus can be defined as \[d(p)=\frac{1}{n}C(p).\]

If $R$ satisfies the assumptions of Theorem \ref{dim_function_exists}, then every projection $p$ is a sum of an orthogonal family of simple projections $p_i,$ $i\in I.$ This fact is proven for $\sims$ in \cite[Propositions 14 and 16, sec. 26]{Berberian}. The theory of types of a Baer $*$-ring also generalizes to any equivalence relation $\sim$ (see \cite[Sections 15--18]{Berberian_web}). The proofs of \cite[Propositions 14 and 16, sec. 26]{Berberian} translate to the general case as well by results of \cite[Sections 15--18]{Berberian_web}. So \cite[Propositions 14 and 16, sec. 26]{Berberian} hold for general $\sim$ as well.

If a projection $p$ is a sum of an orthogonal family of simple projections $p_i,$ $i\in I,$ then the dimension of $p$ can be defined as
\[d(p)=\sum_{i\in I} d(p_i).\]
Details can be found in \cite[Chapter 6]{Berberian}. The proofs from \cite{Berberian} given for $\sims$ hold for a general equivalence relation $\sim$ again thanks to the type decomposition results of \cite[Sections 15--18]{Berberian_web}.

In this way, the dimension of a projection is defined via dimensions of simple projections which, in turn, are defined via central projections. Thus, if a ring has a dimension function, instead of working with all of the projections, one can concentrate on central projections and their simple subprojections (i.e. ``diagonal standard matrix units'') without loosing any relevant information.

We summarize the observations above in the following lemma.

\begin{lemma}
Let $R$ be a ring satisfying assumptions of Theorem \ref{dim_function_exists} with the dimension function $d.$
\begin{enumerate}
\item If $p$ is a simple projection of order $n,$ then $d(p)=\frac{1}{n}C(p).$

\item If $p$ is a projection, then $p$ is a sum of an orthogonal family of simple projections $p_i,$ $i\in I$ and $d(p)=\sum_{i\in I} d(p_i).$
\end{enumerate}
\label{lemma_on_simple}
\end{lemma}
\begin{proof}
(1) follows from the fact that there are orthogonal $p_i,$ $i=1,\ldots, n,$ with $p_i\sim p$ and $C(p)=p_1+\ldots+p_n.$ Thus $C(p)=d(C(p))=d(p_1+\ldots+p_n)=d(p_1)+\ldots+d(p_n)=d(p)+\ldots+d(p)=nd(p).$

(2) The existence of simple projections $p_i$ follows from assumptions of the lemma. The equality $d(p)=\sum_{i\in I} d(p_i)$ holds by property D11.
\end{proof}

\begin{proposition}
If $R$ is a ring satisfying assumptions of Theorem \ref{dim_function_exists} for two equivalence relations on projections $\sim_1$ and  $\sim_2,$ then there are two corresponding dimension functions $d_1$ and $d_2$ and the following holds.
\begin{enumerate}
\item If $\sim_1$ equivalence implies $\sim_2$ equivalence then $d_1=d_2.$

\item If $d_1=d_2$ then the equivalences $\sim_1$ and $\sim_2$ coincide.

\item If  $\sim_1\,\subseteq\, \sim_2,$ then $\sim_1\,=\, \sim_2.$
\end{enumerate}
\label{uniqueness_of_dim}
\end{proposition}

\begin{proof}
(1) If  $\sim_1\,\subseteq\, \sim_2$, then projections that are simple for $\sim_1$ are simple for $\sim_2$ as well. For one such projection $p$ of order $n$, $d_1(p)=\frac{1}{n}C(p)=d_2(p)$ by part (1) of Lemma \ref{lemma_on_simple}.

If $p$ is any projection which is a sum of simple orthogonal projections $p_i,$ $i\in I$ with respect to $\sim_1$, then
$d_1(p)=\sum_{i\in I} d_1(p_i)$ by part (2) of Lemma \ref{lemma_on_simple}. The equality $\sum_{i\in I} d_1(p_i)=\sum_{i\in I}d_2(p_i)$ holds by the previous paragraph. Finally, $\sum_{i\in I} d_2(p_i)=d_2(p)$ by property D11 for $d_2$ and so $d_1(p)=d_2(p).$

(2) Assume that $p\sim_1 q$ for some projections $p$ and $q.$ Then $d_2(p)=d_1(p)=d_1(q)=d_2(q)$ and so $p\sim_2 q$ by property D8 for $d_2.$ The converse follows in the same way.

(3) The claim follows from (1) and (2).
\end{proof}

For rings equipped with a dimension function, we would like to extend it to a dimension defined for all modules. We follow the idea from \cite{Lueck_dimension} used for finite von Neumann algebras. In \cite{Lueck_dimension}, L\"uck generalizes the arguments he used for finite von Neumann algebras to the following more general theorem (which can also be found in \cite{Lueck}).

\begin{theorem}\cite[Theorem 0.6 and Remark 2.14]{Lueck_dimension}
Let $R$ be a ring such that there exists a dimension $\dim$ that assigns to any finitely generated projective
right $R$-module an element of $[0, \infty)$ and such that the
following two conditions hold.
\begin{itemize}
\item[(L1)] If $P$ and $Q$ are finitely generated projective
modules, then
\begin{itemize}
\item[(i)] $\dim(P)$ depends only on the isomorphism class of $P,$

\item[(ii)] $\dim(P\oplus Q)=\dim(P)+\dim(Q).$
\end{itemize}

\item[(L2)] If $K$ is a submodule of finitely generated projective
module $Q,$ then $\cl_{\bnd}(K)$ is a direct summand of $Q$ and
\[\dim(\cl_{\bnd}(K))=\sup\{\dim(P)\;|\; P\mbox { is a fin. gen. proj. submodule of } K\}.\]
\end{itemize}
Then, for every $R$-module $M$, we can define a dimension
\[\dim'(M)=\sup \{ \dim(P)\; |\; P \mbox{ is a fin. gen.
proj. submodule of }M\}\in[0,\infty]\] that satisfies the
following properties:
\begin{enumerate}
\item Extension: $\dim(P)=\dim'(P)$ for every finitely generated
projective module $P.$

\item Additivity: If $\;0\rightarrow M_0\rightarrow M_1\rightarrow
M_2\rightarrow 0$ is a short exact sequence of $R$-modules, then
\[ \dim'(M_1)= \dim'(M_0)+\dim'(M_2).\]

\item Cofinality: If $M = \bigcup_{i \in I}M_i$ is a directed
union, then \[\dim'(M) = \sup\{\;\dim'(M_i)\; |\; i\in I\;\}.\]

\item Continuity: If $K$ is a submodule of a finitely generated
module $M$, then \[\dim'(K)=\dim'(\cl_{\bnd}(K)).\]

\item If $M$ is a finitely generated module, then
\[\dim'(M)=\dim(\unb M)\;\;\mbox{ and }\;\;\dim'(\bnd M)=0.\]

\item The dimension $\dim'$ is uniquely determined by (1) -- (4).
\end{enumerate}
\label{LueckDimension}
\end{theorem}

We intend to use Theorem \ref{LueckDimension} although the value of the dimension of a finitely generated projective module is in the lattice $C_{[0,\infty)}(X),$ the algebra of functions from $C(X)$ with values in $[0,\infty),$ not in $[0,\infty)$ itself. The lattice $C_{[0,\infty)}(X)$ is a boundedly complete lattice with respect to the pointwise ordering (\cite[p. 161--162]{Berberian}) and the proof of Theorem \ref{LueckDimension} still holds if the dimension of a finitely generated projective module is in  $C_{[0,\infty)}(X)$ instead of $[0,\infty).$ As a consequence, the dimension of any module is an element of $C_{[0,\infty)}(X)\cup\{\infty\},$ not in $[0,\infty].$

Also, in the proof of Theorem \ref{LueckDimension}, condition (L1) (i) is only used in the form
\[P\cong Q \Rightarrow \dim(P)=\dim(Q).\]

By Extension property, $\dim'$ can be denoted simply as $\dim$ without any danger of confusion. In what follows, we use just notation $\dim.$

Before we state the next result, recall that two projections $p$ and $q$ are said to be similar if $p=uqu^{-1}$ for some unit $u.$
Recall also that similar projections $p$ and $q$ are algebraically equivalent (since $p=xy$ and $q=yx$ for $x=uq$ and $y=u^{-1}$ in this case).

\begin{proposition}
Let $R$ be a $*$-ring with equivalence relation $\sim$ defined on projections of $M_n(R)$ for every $n.$ Assume that for every $n$ the following holds.
\begin{itemize}
\item[(i)] $M_n(R)$ is a Baer $*$-ring with (Def) to (OA), (GC) and (Fin).

\item[(ii)] Similar projections are equivalent in $M_n(R)$.

\item[(iii)] $M_n(R)$ is regular.
\end{itemize}
Then, dimension $\dim_R$ can be defined for all right $R$-modules by the following two steps:
\begin{enumerate}
\item If $P$ is a finitely generated projective $R$-module,  then
there is a nonnegative integer $n$ and a projection $p\in M_n(R)$ such that the image of left multiplication by $p$ is isomorphic to $P$. In this case, define
\[\dim_R(P)=d(p)\]
where $d$ is the dimension function defined on projections by Theorem \ref{dim_function_exists} using condition (i).

\item If $M$ is any $R$-module, define
\[\dim_R(M)=\sup \{ \dim_R(P)\; |\; P \mbox{ is a
fin. gen. projective submodule of }M\}\] where the supremum on the
right side is an element of $C_{[0,\infty)}(X)$ if it exists and it is a new
symbol $\infty$ otherwise. We define
$a+\infty=\infty+a=\infty=\infty+\infty$ and $a\leq \infty$ for
every $a\in C_{[0,\infty)}(X).$
\end{enumerate}
The function $\dim_R$ satisfies conditions (L1) and (L2) of Theorem \ref{LueckDimension}, and thus $\dim_R$ satisfies properties (1)--(6) of  Theorem \ref{LueckDimension}.
\label{general_dim_on_Q}
\end{proposition}
\begin{remark}
\begin{enumerate}
\item In order to use Theorem \ref{LueckDimension}, we need to impose the conditions of Theorem \ref{dim_function_exists} to all matrix rings over $R$. Assumptions (i)--(iii) guarantee that. Assumption (i) may seem very restrictive at first, but later we show that it holds in most applications. Assumption (ii) clearly holds if $\sim$ is $\sima$ and in all rings in which $\sims$ is equal to $\sima$ (e.g. $AW^*$-algebras). Assumption (iii) is fulfilled in applications when considering the right maximal ring of quotients of a right nonsingular ring.

\item By adding $R$ in the subscript of $\dim$ we emphasize the ring we are working with. This will be essential later on when we shall work simultaneously with modules over a ring $R$ and modules over $\Qrmax(R).$

\item If $p$ is a projection of $M_n(R)$ and if $p(R^n)$ denotes the finitely generated projective module that is the image of left multiplication by $p,$ note that we have
\[\dim_R(p(R^n))=d(p).\]
In general, if $x$ is any element of $M_n(R)$ we let $x(R^n)$ denote the image of the left multiplication map $R^n\rightarrow R^n$ given by $r\mapsto xr.$

In step (1) of Proposition \ref{general_dim_on_Q}, note that for a finitely generated projective $P$ that is a direct summand of $R^n,$ an
idempotent matrix $q$ with image isomorphic to $P$ exists. Choose
$p$ to be the projection such that $p M_n(R)=\ann_r(1-q)=qM_n(R)$ (such $p$ exists because $M_n(R)$ is a Rickart $*$-ring). Thus $p(R^n)=q(R^n).$

\item Note also that the dimension functions on $M_m(R)$ and $M_n(R)$ agree for $m\geq n$ i.e. $d_m\upharpoonleft_{M_n(R)}=d_n$ for all
$m\geq n.$ This is because the dimension function on $R$ is determined by its values on central
projections. The centers of $R$ and $M_n(R)$ are isomorphic under the identification of
diag$(a,a,...,a)\in Z(M_n(R))$ with $a\in Z(R).$ If we identify
diag$(a,a,...,a)\in Z(M_n(R))$ with $na\in Z(R),$ we get the
desired result on the dimension functions.
Considering this and the previous remark, it is foreseeable why $\dim_R$ is well defined for finitely generated projective modules.

\item The fact that the dimension of a finitely generated projective module is not an integer but, in essence, a limit of scalar multiples of certain central projections, should not be considered as a drawback. In fact, such a dimension may be more telling in some situations. For example, all finite $AW^*$-algebras with infinite uniform dimension (e.g. group von Neumann algebras of infinite groups) have the same (infinite) Goldie reduced rank. The dimension that we consider in this paper is finite for those algebras, and thus may help distinguish their different submodules more effectively than the Goldie reduced rank. Similar examples also may be obtained by considering cases when the integer-valued Goldie reduced rank fails to distinguish different uniform modules. For example, consider the ring $R=\Cset\oplus\Cset[x,x^{-1}].$ Both direct summands of $R$ have uniform dimension 1. The dimension we consider distinguish these different summands since $\dim_R(\Cset)=1_{\Cset}$ and $\dim_R(\Cset[x,x^{-1}])=1_{\Cset[x,x^{-1}]}.$
\end{enumerate}
\end{remark}

By Theorem \ref{LueckDimension}, in order to prove Proposition \ref{general_dim_on_Q} it is sufficient to show that (L1) and (L2) hold if $R$ is a ring satisfying assumptions of Proposition \ref{general_dim_on_Q}. We prove that (L1) and (L2) hold in a series of claims collected in the following lemma. The proof of this lemma follows the arguments from \cite{Lia_Baer}.

\begin{lemma} Let $R$ be as in Proposition \ref{general_dim_on_Q}.

\begin{enumerate}
\item $P\cong S\;\; \mbox{ implies }\;\;\dim_R(P)=\dim_R(S),$
for all finitely generated projective
$R$-modules $P$ and $S.$ Moreover, if the equivalent projections of $M_n(R)$ are algebraically equivalent, then  $\dim_R(P)=\dim_R(S)$ implies $P\cong S.$

\item $\dim_R(P\oplus S)=\dim_R(P)+\dim_R(S),$ for all finitely generated projective
$R$-modules $P$ and $S.$

\item If $P$ is a submodule of $R^n,$ then
\[\dim_R(\cl_{\bnd}(P))=\sup\{\,d(p)\;|\; p\in M_n(R) \mbox{ is a projection with }p(R^n)\subseteq P\,\}\]
Thus, if $P$ is finitely generated projective also, then \[\dim_R(P)=\dim_R(\cl_{\bnd}(P)).\]

\item If $P$ and $S$ are finitely generated projective submodules of $R^n,$ then
\[P\subseteq S\mbox{ implies }\;\;\dim_R(P)\leq \dim_R(S).\]

\item If $K$ is a submodule of a finitely generated projective module $S,$ then
\[\dim_R(\cl_{\bnd}(K))=\sup\{\dim_R(P)\;|\; P\mbox { is a fin. gen. projective submodule of } K\}\]
\end{enumerate}
\label{L1L2_lemma}
\end{lemma}

\begin{proof}
(1) Let $P$ and $S$ be two finitely generated projective modules and assume that $P\cong S.$ Let $p$ and $s$ be projections such that
$\dim_R(P)=d(p)$ and $\dim_R(S)=d(s).$ In case that $p$ and $s$ are
matrices of different size, we let $n$ be a positive integer $n$ such that
\[p_n = \left[\begin{array}{cc} p & 0\\
0 & 0 \end{array}\right] \mbox{ and } s_n = \left[\begin{array}{cc} s & 0\\
0 & 0 \end{array}\right] \] are both in $M_n(R)$ and there is an
invertible matrix $u\in M_n(R)$ such that $u p_n=s_n u$ (see \cite[Lemma
1.2.1]{Rosenberg} for details). Thus $p_n$ and $s_n$ are similar. But then $p_n\sim s_n$ by assumption (ii). Thus, $\dim_R(P)=d(p)=d(p_n)=d(s_n)=d(s)=\dim_R(S).$

If we assume that $\sim$ implies $\sima,$ the converse holds as well. Assume that $\dim_R(P)=\dim_R(S).$ Then
$d(p_n)=d(p)=d(s)=d(s_n)$ (we might have to enlarge $p$ and $s$
again). So $p_n\sim s_n.$ Then $p_n\sima s_n$ by assumption and so im$ p_n$ is isomorphic to im$ s_n.$ But
then $P$ is isomorphic to $S.$

(2) Let $P$ and $S$ be finitely generated projective modules with
$p$ and $s$ projections such that $\dim_R(P)=d(p)$ and
$\dim_R(S)=d(s)$ again. Then we can use
\[p\oplus s = \left[\begin{array}{cc} p & 0\\
0 & s \end{array}\right] \] to compute the dimension of $P\oplus
S.$ There is an integer $n$ such that \[
p_n = \left[\begin{array}{cc} p & 0\\
0 & 0 \end{array}\right] \mbox{ and } s_n = \left[\begin{array}{cc} 0 & 0\\
0 & s \end{array}\right] \] are both in $M_n(R).$ Then, $p_n
s_n=s_n p_n=0$ and so $\dim_R(P\oplus S)=d(p\oplus
s)=d(p_n+s_n)=d(p_n)+d(s_n)=d(p)+d(s).$

(3) If $s=\sup\{ p | p\in M_n(R)$ is a projection with $p(R^n)\subseteq
P\},$ first we show that
\[\cl_{\bnd}(P) = s(R^n).\] Note that $\cl_{\bnd}(P)$ is a direct summand of $R^n$ by part (3) of Proposition \ref{VariousClosures}.
Let $r$ denote the projection such that $\cl_{\bnd}(P)=r(R^n).$ We claim that
$s=r.$

If $p$ is any projection with $p(R^n)\subseteq P\subseteq \cl_{\bnd}(P)=r(R^n)$ then $p\leq r.$ Thus $s\leq r.$
For the converse, it is sufficient to show $P\subseteq s(R^n)$ since
then $s(R^n)\supseteq\inf\{q | q\in M_n(R)$ a projection with
$P\subseteq q(R^n) \}(R^n)=\cl_{\bnd}(P)=r(R^n)$ by Proposition
\ref{VariousClosures} and thus we obtain $s\geq r.$ So, let $x\in P.$
Consider a matrix $X\in M_n(R)$ such that the entries in the first
column are coordinates of $x$ in the standard basis and the
entries in all the other columns equal zero. Since $M_n(R)$ is a
regular Rickart $*$-ring, there is a projection
$p_x\in M_n(R)$ such that $X M_n(R)= p_x M_n(R)$ and thus $X(R^n)=p_x(R^n).$ Then $x\in
X(R^n)=x R\subseteq P$ and so we have that $p_x(R^n)\subseteq P$
for all $x\in P.$ So, $p_x\leq s$ for all $x\in P.$ Thus, $x\in
p_x(R^n)\subseteq s(R^n)$ for all $x\in P$ and so $P\subseteq
s(R^n).$

Now it is easy to see that
\[\begin{array}{rcl}\dim_R(\cl_{\bnd}(P)) & = & \dim_R(\sup\{ p | p\in
M_n(R)\mbox{ is a projection with }p(R^n)\subseteq P\}(R^n))\\
& = & d(\sup\{ p | p\in M_n(R)\mbox{ is a projection with
}p(R^n)\subseteq P\})\\
& = & \sup\{ d(p) | p\in M_n(R)\mbox{ is a projection with
}p(R^n)\subseteq P\}\end{array}\] by property (D10) of Theorem
\ref{dim_function_exists}.

The last sentence in part (3) now follows too since $P=\cl_{\bnd}(P)$ for any finitely generated projective $P$ that is a submodule of $R^n.$ The inclusion $\subseteq$ always holds. The converse holds since $R$ is regular, so any finitely generated projective submodule of $R^n$ is a direct summand. Thus $P$ is a direct summand of $R^n$ and then it contains $\cl_{\bnd}(P)$ by parts (1) and (3) of Proposition \ref{VariousClosures}.

(4) Let $p$ be a projection such that $p(R^n)=\cl_{\bnd}(P)$ and
$s$ a projection with $s(R^n)=\cl_{\bnd}(S).$ The inclusion $P\subseteq S$
implies $p(R^n)=\cl_{\bnd}(P)\subseteq \cl_{\bnd}(S)=s(R^n).$
Thus, $p\leq s$ and so $d(p)\leq d(s).$ Hence
\[\dim_R(P)=\dim_R(\cl_{\bnd}(P))=d(p)\leq
d(s)=\dim_R(\cl_{\bnd}(S))=\dim_R(S)\] by part (3).

(5)  Since $S$ is finitely generated projective, there is a nonnegative integer $n$
such that $S$ is a direct summand of $R^n.$ The module $\cl^S_{\bnd}(K)$ is a
direct summand of $S$ by Theorem \ref{SplittingOfBnd}. Thus, it is a direct summand of $R^n$ as well and so
$\cl^{R^n}_{\bnd}(K)\subseteq \cl^S_{\bnd}(K)$ by Proposition
\ref{VariousClosures}. Since $S\subseteq R^n$ implies
$\cl^S_{\bnd}(K)\subseteq \cl^{R^n}_{\bnd}(K),$ we have that $\cl^S_{\bnd}(K) = \cl^{R^n}_{\bnd}(K).$ So we can consider $R^n$ instead of $S.$

\[\begin{array}{rcl}
\dim_R(\cl^{R^n}_{\bnd}(K)) & = & \sup\{d(p)\;|\; p\mbox { is a
projection in }M_n(R)\mbox{ with }p(R^n)\subseteq K\}\\
& \leq & \sup\{\dim_R(P)\;|\; P\mbox { is a fin. gen. proj.
submodule of } K\}\end{array}\] where the first equality holds by part (3).

Conversely,
\[\begin{array}{lcr}
\sup\{\dim_R(P)\;|\; P\mbox { is a fin. gen. projective submodule
of } K\} & \leq & \\
\sup\{\dim_R(P)\;|\; P\mbox { is a fin. gen. projective submodule
of } \cl^{R^n}_{\bnd}(K)\} & = & \\
\dim_R(\cl^{R^n}_{\bnd}(K)). & &
\end{array}\]
For the last equality: $\leq$ holds by monotony for the dimensions of
finitely generated projective modules by part (4). The converse follows
since $\cl^{R^n}_{\bnd}(K)$ is finitely generated projective by
Theorem \ref{SplittingOfBnd}.
\end{proof}

\begin{proof}{\it [Proposition \ref{general_dim_on_Q}]}
The dimension function $d$ on $M_n(R)$ exists by Theorem \ref{dim_function_exists}. It defines the dimension $\dim_R$ that satisfies properties (L1) and (L2) by Lemma \ref{L1L2_lemma}. Thus, Proposition \ref{general_dim_on_Q} holds by Theorem \ref{LueckDimension}.
\end{proof}

We prove the main result of this section now. Assumptions of this result may again look severe, but they hold for the class of positive definite strongly semihereditary $*$-rings 
as Theorem \ref{strongly_semihereditary_has_dim} will demonstrate. Also, note that we do not assume the ring in Theorem \ref{general_dim_on_R} to be a Baer $*$-ring or to satisfy any of Kaplansky's axioms. In fact, we do not assume it to be involutive (though we assume that the regular overring is involutive).

\begin{theorem}
Let $R$ be a right semihereditary ring such that $M_n(R)$ is Baer for every $n$ (alternatively, it is sufficient to assume that $R$ is right semihereditary such that $\cl_{\bnd}(K)$ is a direct summand of $P$ for every $K\leq P$ and $P$ finitely generated projective).
Let $R$ embed in a ring $Q$ such that the following holds.
\begin{itemize}
\item[(a)] $Q$ satisfies assumptions of Proposition \ref{general_dim_on_Q}.
\item[(b)] There is a correspondence $\mu$ between finitely generated projective modules of $R$ and $Q$ such that for every finitely generated projective $R$-module $P,$ $\mu$ defines a bijection of the set of the direct summands of $P$ onto the set of the direct summands of $\mu(P).$ Moreover, for any finitely generated submodule $K$ of $P,$ $\mu(K)=\mu(\cl_{\bnd}(K)).$
\item[(c)] The map $\mu$ preserves taking isomorphic images, submodules and direct sums and $\mu(R)=Q.$
\item[(d)] By part (b), there is an inverse map $\mu^{-1}$ defined on direct summands of every finitely generated projective $Q$-module. The map $\mu^{-1}$ preserves taking submodules.
\end{itemize}
Then, dimension $\dim_R$ can be defined for all finitely generated projective $R$-modules $P$ by
\[\dim_R(P)=\dim_Q(\mu(P))\]
where $\dim_Q$ is dimension on $Q$ that can be defined by assumption (a). The dimension $\dim_R$
can be extended to all $R$-modules by
\[\dim_R(M)=\sup \{ \dim_R(P)\; |\; P \mbox{
fin. gen. projective submodule of }M\}.\]
The function $\dim_R$ satisfies conditions (L1) and (L2) of Theorem \ref{LueckDimension}, and thus $\dim_R$ satisfies properties (1)--(6) of  Theorem \ref{LueckDimension}.
\label{general_dim_on_R}
\end{theorem}
\begin{proof}
To prove the theorem it is again sufficient to check that conditions (L1) and (L2) hold. The property (L1) is satisfied by  assumptions (a) and (c).

To check that (L2) holds, let $K$ be a submodule of a finitely generated projective module $S$. If $S$ is a direct summand of $R^n$ for a nonnegative integer $n,$ we can use the same argument as in the proof of part (5) of Lemma \ref{L1L2_lemma} to show that
$\cl^S_{\bnd}(K)=\cl^{R^n}_{\bnd}(K).$  Hence, we can work in $R^n.$ In the rest of the proof, we let
$\cl_{\bnd}$ denote $\cl^{R^n}_{\bnd}.$

Let $s$ be the supremum of an increasingly directed family of projections $p$ in $M_n(Q)$ such that $ p(Q^n)=\mu(P)$ for all finitely generated projective modules $P$ that are submodules of $K.$ Let $S'=s(Q^n).$ Note that $d(s)=\sup d(p)$ by property (D10) and so $\dim_Q(S')=d(s)=\sup d(p) =\sup\dim_Q(\mu(P))=\sup\dim_R(P)$ where the supremum is taken over all finitely generated projective modules $P\leq K.$ Every such submodule $P$ determines a direct summand $\mu(\cl_{\bnd}(P))$ of $Q^n$ that is a submodule of $\mu(\cl_{\bnd}(K)).$ Note that here we are using assumption (c). Thus, $\dim_Q(S')=\sup\dim_R(P)\leq \sup\dim_Q(T)$ where the supremum is taken over all finitely generated projective modules $T$ in $Q^n$ with $T\leq \mu(\cl_{\bnd}(K)).$ This last supremum is equal to $ \dim_Q(\mu(\cl_{\bnd}(K)))$ by definition of $\dim_Q$ and $\dim_Q(\mu(\cl_{\bnd}(K)))=\dim_R(\cl_{\bnd}(K))$ by definition of $\dim_R.$ Thus,
\[\dim_R(\cl_{\bnd}(K))\geq\sup\{\dim_R(P)\;|\; P\mbox { is a fin. gen. projective submodule of } K\}\]

To prove the converse, let $s_K$ be the supremum of projections $p_x\in M_n(Q)$ such that $p_x(Q^n)=\mu(xR)$ for $x\in K.$ Note that for every $x\in K,$ the finitely generated module $xR$ is projective since $R$ is right semihereditary. Let $S_K=s_K(Q^n).$

By construction, $S_K\leq S'.$ For $x\in K,$ $\mu(\cl_{\bnd}(xR))=\mu(xR)\leq S_K$ by assumption (b). Assumption (d) then ensures that $\cl_{\bnd}(xR)\leq \mu^{-1}(S_K).$ Hence $x\in xR\leq \cl_{\bnd}(xR)\leq \mu^{-1}(S_K)$ for every $x\in K.$ Thus, $K\leq \mu^{-1}(S_K).$ But then $\cl_{\bnd}(K)\leq \mu^{-1}(S_K)$ by Proposition \ref{VariousClosures}.
We obtain then that $\mu(\cl_{\bnd}(K))\leq S_K\leq S'$ by assumption (c). Thus
\[\dim_R(\cl_{\bnd}(K))=\dim_Q(\mu(\cl_{\bnd}(K)))\leq \dim_Q(S')=\sup\{\dim_R(P)| P\leq K\mbox{ fin. gen. proj.}\}.\]
\end{proof}

\begin{definition}
If $R$ satisfies assumptions of Theorem \ref{LueckDimension} with $[0, \infty)$ possibly replaced by the algebra of continuous non-negative functions $C_{[0, \infty)}(X)$ on a Stonian space $X,$ we say that $R$ {\em has dimension for $\sim$}.
\end{definition}

\begin{corollary} Let $R$ satisfy assumptions of Theorem \ref{general_dim_on_R} and dimension $\dim_R$ on $R$ be defined via dimension $\dim_Q$ on a regular overring $Q$ using a map $\mu$ as in Theorem \ref{general_dim_on_R}. If the map $\mu$ extends to a functor on all $R$-modules that is exact and agrees with direct limits, then
\[\dim_Q(\mu(M))=\dim_R(M)\]
for any right $R$-module $M.$
\label{dimQ_versus_dimR}
\end{corollary}
\begin{proof} If $M$ is a finitely generated projective module,
this follows from Theorem \ref{general_dim_on_R}.

If $M$ is a submodule of any finitely generated projective
$R$-module, $M$ is a directed union of its finitely generated
modules $M_i,$ $i\in I$ that are projective since $R$ is right semihereditary. The modules $\mu(M_i)$ are finitely generated projective submodules of $\mu(M)$ by exactness of $\mu$. Thus,
\[\dim_R(M)=\sup_{i\in I} \dim_R(M_i)= \sup_{i\in I} \dim_Q(\mu(M_i))=\dim_Q(\mu(M))\] by cofinality of $\dim_R$ and $\dim_Q$
and the fact that $\mu$ commutes with direct limits.

If $M$ is a finitely generated $R$-module, then $M$ is a quotient
of some finitely generated projective module $P$ and its submodule $K$. Then,
$\dim_R(M)=\dim_R(P)-\dim_R(K)=\dim_Q(\mu(P))-\dim_Q(\mu(K))=\dim_Q(\mu(M))$ by additivity of
dimensions $\dim_R$ and $\dim_Q,$ by exactness of $\mu$ and by the previous step.

Finally, let $M$ be an arbitrary $R$-module. Consider $M$ as a directed union of its finitely generated submodules $M_i.$  By using cofinality of dimension and the previous case, we obtain that
$\dim_R(M)=\sup_{i\in I} \dim_R(M_i)= \sup_{i\in I} \dim_Q(\mu(M_i))=\dim_Q(\mu(M)).$
\end{proof}

\section{Dimension of strongly semihereditary $*$-rings}
\label{section_dimension_strongly_semihereditary}

In this section, first we show that all positive definite strongly semihereditary $*$-rings 
 have dimension for $\sima.$ We obtain this dimension by Theorem \ref{general_dim_on_R} with $\sim$ interpreted as $\sima$ and by using favorable properties of involutive strongly semihereditary rings from Section \ref{section_involutive_strongly_semihereditary}.

\begin{theorem}
Every strongly semihereditary $*$-ring with a positive definite involution has dimension for $\sima$.
\label{strongly_semihereditary_has_dim}
\end{theorem}
\begin{proof}
Let $R$ be a strongly semihereditary ring with a positive definite involution. Let $Q=\Qrmax(R).$ Thus $M_n(Q)=\Qrmax(M_n(R))$ is a regular ring since $M_n(R)$ is right nonsingular. Hence assumption (iii) of Proposition \ref{general_dim_on_Q} holds. Assumptions (ii) of Proposition \ref{general_dim_on_Q} is trivially satisfied since we consider $\sima.$

Let us check that assumption (i) of Proposition \ref{general_dim_on_Q} holds as well. Note that $R$ is a strongly semihereditary $*$-ring with $n$-proper involution and so $M_n(Q)$ is a Baer $*$-ring by part (6) of Proposition \ref{properties}. Recall that (Def), (CC), (IP) and (FA) hold for $\sima.$  The ring  $M_n(Q)$ is regular and so (OA) and (P) hold as well (see \cite[p. 47]{Kaplansky}). Thus (GC) and (Add) hold also (by \cite[Theorem 2.1]{Maeda_Holland}). Since $M_n(Q)$ is regular and self-injective, then (Fin) holds for $\sima$ as well (see \cite[Prop. 5.2 and Thm. 9.29]{Goodearl_book}).

Thus, $Q$ is a ring that has dimension for $\sima$ by Proposition \ref{general_dim_on_Q}.

Then, let us check that $R$ satisfies assumptions of Theorem \ref{general_dim_on_R}. First, note that $R$ is right semihereditary such that $M_n(R)$ is Baer for every $n.$ Furthermore, we have just seen that assumption (a) holds.

If $P$ is a finitely generated projective $R$-module, let us consider $\mu$ to be the map $S\mapsto S\otimes_R Q=E(S)$ for $S$ a direct summand of $P.$ By Proposition \ref{K0_theorem} this defines a bijective correspondence between direct summands of $P$ and $\mu(P).$ If $P$ is a direct summand of $R^n$, and $K$ is a finitely generated projective submodule of $P,$ it is sufficient to consider $K$ as a submodule of $R^n$ (using the same reasoning we noted in the proof of Theorem \ref{general_dim_on_R}).

We claim that $\mu(K)=K\otimes_R Q$ is equal to $\mu(\cl_{\bnd}(K))=\cl_{\bnd}(K)\otimes_R Q.$ The module $(\cl_{\bnd}(K)\otimes_R Q)/(K\otimes_R Q)$ is a finitely presented $Q$-module. Since $Q$ is regular, it is finitely generated projective. So, it is an unbounded $Q$-module. We claim it is a bounded module as well and hence has to be trivial. To show that $(\cl_{\bnd}(K)\otimes_R Q)/(K\otimes_R Q)$ is bounded, it is sufficient to show that $\cl_{\bnd}(K\otimes_R Q)=\cl_{\bnd}(K)\otimes_R Q.$ To prove one inclusion, note that $\cl_{\bnd}(K)\otimes_R Q$ is a direct summand of $Q^n$ that contains $K\otimes_R Q.$ Thus, it contains the smallest direct summand of $Q^n$ that contains $K\otimes_R Q$ and so we have $\cl_{\bnd}(K\otimes_R Q)\leq \cl_{\bnd}(K)\otimes_R Q.$

To show the converse, let $T$ be a direct summand of $Q^n$ that contains $K\otimes_R Q.$ Then $T\cap R^n$ is a direct summand of $R^n$ and it contains $(K\otimes_R Q)\cap R^n=\cl_{\T}(K)=\cl_{\bnd}(K)$ by Proposition \ref{K0_theorem}, Example \ref{Goldie_example} and Proposition \ref{VariousClosures}. Thus, $\cl_{\bnd}(K)\otimes_R Q$ is contained in $T=(T\cap R^n)\otimes_R Q.$ So, $\cl_{\bnd}(K)\otimes_R Q$ is contained in all direct summands of $Q^n$ that contain $K\otimes_R Q$ and so $\cl_{\bnd}(K)\otimes_R Q\leq \cl_{\bnd}(K\otimes_R Q).$ This shows that assumption (b) holds.

Assumption (c) holds by definition of the map $\mu$ and assumption (d) holds by the definition of the map $\mu^{-1}$ given by Proposition \ref{K0_theorem}.
\end{proof}

\begin{corollary} If $R$ is a strongly semihereditary ring with a positive definite involution and dimension $\dim^a_R$ for $\sima$, and $Q$ is $\Qrmax(R)$ with dimension $\dim^a_Q$ for $\sima,$ then
\[\dim^a_Q(M\otimes_R Q)=\dim^a_R(M)\]
for any right $R$-module $M.$
\label{dimQ_versus_dimR_for_stronglysemiher}
\end{corollary}
\begin{proof} This is a direct corollary of Theorem \ref{strongly_semihereditary_has_dim} and Corollary \ref{dimQ_versus_dimR} since right tensoring with $Q$ defines a functor that is exact and agrees with direct limits.
\end{proof}

As a corollary of Theorem \ref{strongly_semihereditary_has_dim}, we obtain dimension for $\sima$ of noetherian Leavitt path algebras over positive definite fields.

\begin{corollary} Let $K$ be a positive definite field and $E$ a finite no-exit graph. Then the Leavitt path algebra $L_K(E)$ has dimension for $\sima.$
\label{LPAs}
\end{corollary}

On first glance, there may seem to be no hope for any reasonable dimension function defined via projections of a $*$-ring that is not Rickart $*$ (as, for example, a noetherian Leavitt path algebra can be by Example \ref{examples_strongly_semihereditary}). This is because idempotents cannot always be replaced by projections, so a finitely generated projective module cannot always be associated to a projection, just to an idempotent. However, Theorem \ref{strongly_semihereditary_has_dim} and Corollary \ref{LPAs} demonstrate that even non Rickart $*$-rings can have a well-behaved dimension. This fact emphasizes the value of Theorem \ref{strongly_semihereditary_has_dim}.

Next, we consider the class of Baer $*$-rings studied in \cite{Lia_Baer}. This class is defined using the nine axioms below.
\begin{itemize}
\item[(A1)] $R$ is finite (i.e. satisfies (Fin) for $\sims$).

\item[(A2)]  For every $0\neq x\in R,$ there exists a self-adjoint
$y$ in the double commutator $\{x^*x\}''$ of $x^*x$ such that $(x^*x)y^2$ is a nonzero projection. This property is called the {\em existence of projections axiom.}

\item[] For every $x\in R$ such that $x=x_1^*x_1+x_2^*x_2+\ldots
+x_n^*x_n$ for some $n$ and some $x_1,x_2,\ldots,x_n\in R$ (such
$x$ is called {\em positive}), there is a unique $y\in\{x^*x\}''$
such that $y^2=x$ and $y$ positive. This property is called the {\em unique positive square root axiom.}

\item[(A3)] Partial isometries are addable (for precise formulation, see \cite[Section 11]{Berberian}).

\item[(A4)] For all $x\in R,$ $1+x^*x$ is
invertible. In this case, we say that $R$ is {\em symmetric}.

\item[(A5)] There is a central element $i\in R$ such that $i^2=-1$
and $i^*=-i.$

\item[(A6)] For each unitary $u\in R$ such that RP$(1-u)=1,$ there exists an
increasingly directed sequence of projections $p_n\in\{u\}''$ with
supremum 1 such that $(1-u)p_n$ is invertible in $p_n Rp_n$ for
every $n.$ This property is called {\em unitary spectral axiom.}

\item[(A7)] If $p_n$ is orthogonal sequence of projections with supremum 1 and
$a_n\in R$ such that $0\leq a_n\leq p_n,$ then there is $a\in R$
such that $a p_n=a_n$ for all $n.$ This property is called {\em positive sum axiom.}

\item[(A8)] The axiom (P) holds for all projections in $M_n(R)$ for every $n.$

\item[(A9)] Every sequence of orthogonal projections in $M_n(R)$
has a supremum.
\end{itemize}

All finite $AW^*$-algebras (thus finite von Neumann algebras as well) satisfy these axioms.

In \cite[Chapter 8]{Berberian}, it is shown that if $R$ is a Baer $*$-ring satisfying (A1)--(A7), then there is a regular Baer $*$-ring $Q$ satisfying (A1) -- (A7) such that $R$ is $*$-isomorphic to a $*$-subring of $Q$, all projections, unitaries and partial isometries of $Q$ are in $R,$ and $Q$ is unique up to $*$-isomorphism. Moreover, the projections, unitaries and partial isometries of $M_n(Q)$ and $M_n(R)$ are the same (\cite[Section 56, Prop. 3]{Berberian}). In \cite[Chapter 9]{Berberian}, it is shown that if $R$ satisfies also (A8) and (A9), then $M_n(R)$ is a finite Baer $*$-ring with (GC) for every $n.$ In \cite{Lia_Baer} it is shown that Baer $*$-rings with (A1)--(A9) have dimension for $\sims.$

In \cite[Theorem 4]{Lia_nine}, it is shown that (A1)--(A7) imply (A9) and that $M_n(R)$ is a Baer $*$-ring for all $n.$ Even without (A9), the rather restrictive axiom (A8) remains. We show that rings with (A1)--(A7) have dimension for $\sima$ and that the matrix ring $M_n(R)$ is finite for every $n.$ Thus, it is not necessary to assume (A8) in \cite[Section 58]{Berberian}. If $R$ is a Baer $*$-ring that satisfies (A8) in addition to (A1)--(A7), then $\sims$ and $\sima$ are the same relation on $M_n(R)$ and, as a consequence, dimension for $\sims$ from \cite{Lia_Baer} coincides with dimension for $\sima.$

\begin{corollary}
Let $R$ be a Baer $*$-ring with (A1)--(A7).
\begin{enumerate}
\item $M_n(R)$ is a finite Baer $*$-ring for every $n$ and $R$ has dimension $\dim^a$ for $\sima.$

\item If $R$ satisfies (A8), then $R$ has dimension $\dim^*$ for $\sims.$ In this case, $\sims\; =\; \sima$ on $M_n(R)$ and $\dim^*=\dim^a.$
\end{enumerate}
\label{VNA-like}
\end{corollary}
\begin{proof}
(1) In \cite[Proposition 3]{Lia_Baer}, it is shown that the regular ring $Q$ of a ring $R$ with (A1)--(A7) is both (left and right) maximal and classical ring of quotients of $R$. Since the total ring of quotients is contained in maximal and contains the classical ring of quotients for an Ore ring, $Q$ is also the left and right total ring of quotients. Moreover, $R$ is semihereditary (see \cite[Corollary 5]{Lia_Baer}) and so $R$ is an involutive strongly semihereditary ring. The involution on $R$ is positive definite since $M_n(R)$ is a Rickart $*$-ring (by \cite[Section 56, Theorem 1]{Berberian}) and, thus, its involution is proper. Also, the ring $M_n(R)$ is a Baer $*$-ring since it is Baer and a Rickart $*$-ring. It is finite by Proposition \ref{properties}. The dimension $\dim^a$ for $\sima$ exists by Theorem \ref{strongly_semihereditary_has_dim}.

(2) The regular ring $M_n(Q)$ satisfies assumptions of Theorem \ref{dim_function_exists} for $\sims\,$: (Def) to (OA) always holds for $\sims,$ (P) is assumed to hold for $\sims$ in $M_n(R)$ by (A8) (thus $M_n(Q)$ as well since the projections are the same), so (GC) holds as well. (Fin) also holds in $M_n(Q)$ since it is $*$-regular. Thus, $M_n(Q)$ has a dimension function $d^*$ for $\sims$.

The ring $M_n(Q)$ also satisfies (Def) to (OA), (GC) and (Fin) for $\sima$ so the dimension function $d^a$ exists as well. Since $\sims$ implies $\sima,$ the dimension function $d^*$ is equal to the dimension function $d^a$ and $\sims\, =\, \sima$ on $M_n(Q)$ by Proposition \ref{uniqueness_of_dim}. The dimension $\dim^*_Q$ (that exists by results from \cite{Lia_Baer}) coincides with $\dim^a_Q$ (that exists by Proposition \ref{general_dim_on_Q}).

The equivalences $\sims$ and $\sima$ coincide on $M_n(R)$ also since the projections of $M_n(R)$ and $M_n(Q)$ are the same. The dimension $\dim^*_R$ (that exists by results from \cite{Lia_Baer}) coincides with $\dim^a_R$ (that exists by Theorem \ref{general_dim_on_R}).
\end{proof}

Note that $\sims\; =\; \sima$ on $R$ follows already from (A2) and (A3) (see \cite[Exercise 8A, p. 9]{Berberian}).
However, Corollary \ref{VNA-like} guarantees that this happens on $M_n(R)$ as well if (A1)--(A8) are assumed.

We have seen that the restrictive axioms (A8) and (A9) are not necessary for the main result of \cite[Section 58]{Berberian}: just (A1)--(A7) are sufficient in order to obtain that $M_n(R)$ is a finite Baer $*$-ring. The axiom (A9) can be completely dropped by \cite{Lia_nine}. In \cite[Chapter 9]{Berberian}, (A8) is used in order to obtain (GC) on $M_n(R)$ for every $n.$ We wonder whether (A8) is necessary for this last result and whether it follows from (A1)--(A7).

\end{document}